\documentclass[11pt]{amsart}
\usepackage[utf8]{inputenc}
\usepackage[T1]{fontenc}
\usepackage{amsmath,amssymb,amsthm,mathtools}
\usepackage[authoryear,round]{natbib}
\usepackage{graphicx,booktabs,xcolor}
\usepackage[margin=2.8cm]{geometry}
\usepackage[colorlinks=true,linkcolor=blue!50!black,citecolor=blue!50!black,urlcolor=blue!50!black]{hyperref}

\newtheorem{theorem}{Theorem}[section]
\newtheorem{proposition}[theorem]{Proposition}
\newtheorem{lemma}[theorem]{Lemma}
\newtheorem{corollary}[theorem]{Corollary}
\theoremstyle{definition}
\newtheorem{definition}[theorem]{Definition}
\newtheorem{example}[theorem]{Example}
\theoremstyle{remark}
\newtheorem{remark}[theorem]{Remark}

\newcommand{\C}{\mathcal{C}}
\newcommand{\CSI}{\mathcal{C}_{\mathrm{SI}}}
\newcommand{\K}{\mathcal{K}}
\newcommand{\Proj}{\mathcal{P}}
\newcommand{\I}{\mathcal{I}}
\newcommand{\R}{\mathbb{R}}
\newcommand{\E}{\mathbb{E}}
\newcommand{\Prob}{\mathbb{P}}
\newcommand{\1}{\mathbf{1}}
\newcommand{\xiu}{\xi^{\uparrow}}
\newcommand{\xid}{\xi^{\downarrow}}
\newcommand{\omu}{\omega^{\uparrow}}
\newcommand{\Var}{\operatorname{Var}}
\newcommand{\Cov}{\operatorname{Cov}}
\newcommand{\lcm}[1]{\widehat{#1}}
\newcommand{\ip}[2]{\langle #1,#2\rangle}

\begin{document}

\title{The monotone part of Chatterjee's rank correlation}

\author{Karl Friedrich Siburg}
\address{Fakult\"at f\"ur Mathematik, Technische Universit\"at Dortmund, 44227 Dortmund, Germany}
\email{karlfriedrich.siburg@tu-dortmund.de}

\date{}

\begin{abstract}
Chatterjee's rank correlation $\xi$ measures how strongly a random variable $Y$ depends functionally on a random variable $X$, but it does not distinguish increasing, decreasing and non-monotone relationships. We ask how much of the dependence measured by $\xi$ is explained by an increasing relationship. For continuous marginals, $\xi$ is, up to a factor, a squared distance from independence in a Hilbert space geometry derived from the Sobolev scalar product for copulas. We answer it by the metric projection onto the closed convex set of stochastically increasing copulas in this geometry. The projection is explicit: it replaces every vertical section by its least concave majorant, and it coincides with the smallest stochastically increasing copula dominating the given one. It splits $\xi$ orthogonally into a monotone part $\xiu$ and the squared distance to the class of stochastically increasing copulas; dually, $\xiu$ is, up to a factor, the squared distance to the polar cone, which corresponds to negative quadrant dependence. The monotone part vanishes exactly for negatively quadrant dependent copulas, equals $\xi$ exactly for stochastically increasing copulas, and equals one only for the comonotonicity copula. In contrast to $\xi$, it is increasing in the concordance order and continuous with respect to weak convergence. Symmetrisation decomposes the measure of mutual complete dependence in the same way. Finally, the construction extends to arbitrary distributions via isotonic regression of conditional distribution functions. In particular, the plug-in estimator based on the empirical distribution is strongly consistent for every copula, without smoothing or tuning parameters.
\end{abstract}


\maketitle

\section{Introduction}\label{sec:intro}

For a random variable $X$ and a non-degenerate random variable $Y$ with distribution $\mu$, Chatterjee's rank correlation is
\begin{equation}\label{eq:xi-general}
\xi(Y|X)=\frac{\int\Var\bigl(\Prob(Y\ge t\mid X)\bigr)\,d\mu(t)}{\int\Var\bigl(\1\{Y\ge t\}\bigr)\,d\mu(t)}
\end{equation}
\citep{Chatterjee2021,AzadkiaChatterjee2021}. It takes values in $[0,1]$, it vanishes if and only if $X$ and $Y$ are independent, and it equals one if and only if $Y=f(X)$ almost surely for some measurable function $f$. For continuous marginals, $\xi$ depends only on the copula $C$ of $(X,Y)$ and coincides with the measure of regression dependence of \citet{DetteSiburgStoimenov2013},
\begin{equation}\label{eq:xi-copula}
\xi(C)=6\int_0^1\!\!\int_0^1\bigl(\partial_1C(u,v)\bigr)^2\,du\,dv-2 .
\end{equation}

\emph{The question.} By design, $\xi$ does not distinguish between increasing, decreasing and non-monotone functional relationships; for instance, $\xi(M)=\xi(W)=1$ for the Fr\'echet--Hoeffding bounds $M$ and $W$. This paper addresses the question
\begin{center}
\emph{How much of the dependence measured by $\xi$ is explained by an increasing relationship?}
\end{center}
Equivalently, how much of the variation of $\1\{Y\ge t\}$ that $X$ explains can already be explained by non-decreasing functions of $X$? Concordance measures such as Spearman's $\rho_S$ do not answer it: they are not related to $\xi$ by any decomposition, and they balance positive against negative association.

\emph{Two further features of $\xi$.} As a consequence of its design, $\xi$ is not monotone with respect to the concordance (pointwise) order of copulas. Moreover, $\xi$ is not continuous with respect to weak convergence \citep{BucherDette2025}, although it is continuous with respect to Markov products \citep{AnsariFuchs2025cont}. None of this is a shortcoming, since $\xi$ is designed to detect any kind of functional dependence. Both properties, however, hold for the monotone part introduced below.

\emph{Geometric interpretation.} \citet{SiburgStoimenov2008} introduced the Sobolev scalar product $\ip{A}{B}=\int\nabla A\cdot\nabla B\,d\lambda$ on the linear span of bivariate copulas. It splits into the two partial scalar products $\ip{A}{B}_i=\int\partial_iA\,\partial_iB\,d\lambda$, $i=1,2$, with associated distances $d_1,d_2$. As recalled in Section~\ref{sec:prelim},
\begin{equation}\label{eq:geom}
\xi(C)=6\,d_1(C,\Pi)^2,\qquad
\omega(C)^2:=3\,d(C,\Pi)^2=\tfrac12\bigl(\xi(C)+\xi(C^\top)\bigr),
\end{equation}
where $\Pi$ is the independence copula, $C^\top$ the transposed copula and $\omega$ the measure of mutual complete dependence of \citet{SiburgStoimenov2010}. Thus $\xi$ is a squared distance in a Hilbert space geometry, and its symmetrisation is the squared Sobolev distance.

\emph{Approach and main results.} We address this question geometrically. The stochastically increasing (SI) copulas, i.e., the copulas of positively regression dependent random vectors in the sense of \citet{Lehmann1966}, form a closed convex set $\CSI$. We project onto $\CSI$ with respect to the distance $d_1$ in which $\xi$ is a squared distance. Our main results are as follows.
\begin{enumerate}
\item The metric projection $\Proj C$ exists and is explicit: each vertical section $C(\cdot,v)$ is replaced by its least concave majorant. Equivalently, $\partial_1\Proj C(\cdot,v)$ is the isotonic regression of $\partial_1C(\cdot,v)$ onto non-increasing functions. Moreover, $\Proj C$ is the smallest SI copula dominating $C$ pointwise (Theorem~\ref{thm:proj}).
\item In the tangent space at $\Pi$, the directions of stochastic increasingness and those of negative quadrant dependence form mutually polar cones, $\mathfrak K$ and $\mathfrak K^\circ$. Accordingly, $C-\Pi=(\Proj C-\Pi)+(C-\Proj C)$ is the Moreau decomposition of $C-\Pi$ into its projections onto $\mathfrak K$ and $\mathfrak K^\circ$, and the two parts are orthogonal (Theorem~\ref{thm:polar}).
\item The \emph{monotone part} $\xiu(C):=\xi(\Proj C)$ satisfies the exact decomposition $\xi(C)=\xiu(C)+6\,d_1(C,\CSI)^2$. It vanishes if and only if $C$ is negatively quadrant dependent, it equals $\xi(C)$ if and only if $C$ is SI, and it equals one if and only if $C=M$. It is increasing in the concordance order, Lipschitz continuous with respect to $d_1$ after taking square roots, and continuous with respect to uniform (equivalently, weak) convergence of copulas (Theorem~\ref{thm:xiu}). By the polarity in (2), $\xiu(C)/6$ and $(\xi(C)-\xiu(C))/6$ are the squared distances of $C-\Pi$ to the cones $\mathfrak K^\circ$ and $\mathfrak K$, respectively; this explains why NQD and SI are exactly the cases $\xiu=0$ and $\xiu=\xi$.
\item The symmetrisation $\omu$ decomposes the measure of mutual complete dependence, $\omega^2=(\omu)^2+3\,[\,d_1(C,\CSI)^2+d_2(C,\CSI^\top)^2\,]$, and inherits all properties above (Theorem~\ref{thm:sym}).
\item The construction extends to arbitrary distributions, including discrete and mixed ones, by projecting the conditional distribution functions in~\eqref{eq:xi-general} onto non-decreasing functions of $X$. All characterisations remain valid (Theorem~\ref{thm:general}). In particular, the plug-in estimator $\xiu(\Prob_n)$ based on the empirical distribution is strongly consistent for every copula, without smoothing or tuning parameters, whereas the plug-in principle fails for $\xi$ (Theorem~\ref{thm:est}).
\end{enumerate}
Combining $\xiu$ with its decreasing counterpart $\xid$ yields $\max\{\xiu,\xid\}$, a weakly continuous measure that vanishes exactly under independence and equals one exactly under perfect monotone dependence (Corollary~\ref{cor:max}).

\emph{What is new.} Existing measures of monotone dependence are of concordance type. Spearman's $\rho_S$, Kendall's $\tau$ and the asymmetric grade correlation of \citet{WalzEberlGneiting2025}, which reduces to $\rho_S$ for continuous marginals, attain their maximal value exactly under increasing functional dependence. However, they balance positive against negative association, so that their zero set neither characterises independence nor negative quadrant dependence, and they carry no information about the strength of non-monotone functional dependence. Measures of functional dependence such as $\xi$, the measure $\zeta_1$ of \citet{Trutschnig2011} or the rearranged dependence measures of \citet{StrothmannDetteSiburg2024} detect any functional relationship, but they assign the maximal value to every functional relationship, whether increasing, decreasing or non-monotone. The monotone part $\xiu$ is not meant as yet another coefficient of increasing functional dependence; for that purpose $\rho_S$ and $\tau$ suffice. Rather, it answers the question raised above, which neither type of measure addresses. Accordingly, $\xiu$ is an exact component of $\xi$ itself, it vanishes exactly under negative quadrant dependence, it equals $\xi$ exactly under stochastic increasingness, and it attains one exactly under increasing functional dependence. It is linked to concordance by the sharp bound $\xiu\ge\frac12(\rho_S)_+^2$ (Theorem~\ref{thm:xiu}(g)), so positive concordance forces a positive monotone part. The converse fails: the mixture $\frac12M+\frac12W$ and the copula of $(U,|2U-1|)$ have $\rho_S=0$, but $\xiu\approx0.079$ and $\xiu\approx0.318$, respectively (Table~\ref{tab:cmp}). Behind the decomposition lies the structural fact (2), which appears to be new: in the geometry of $\xi$, stochastic increasingness and negative quadrant dependence are dual to each other in the sense of convex cones. To the best of our knowledge, neither the projection onto $\CSI$ in the geometry of $\xi$ nor a decomposition of $\xi$ into an increasing part and a remainder has been considered before; recent work relates $\xi$ to positive dependence and to concordance through inequalities and exact regions \citep{AnsariRockel2026,Rockel2026,OrendayLaresRockel2026}, not through a decomposition.

\emph{Statistical relevance.} The properties of $\xiu$ matter for statistical practice in four ways. First, the identity $1=\xiu+(\xi-\xiu)+(1-\xi)$ is an orthogonal decomposition of the variation of the indicators $\1\{Y\ge t\}$ into a part explained by non-decreasing functions of $X$, a part explained by $X$ but not by non-decreasing functions of $X$, and an unexplained part (Remark~\ref{rem:anova}). Thus $\xiu$ is an isotonic counterpart of a coefficient of determination. The middle part contains decreasing as well as genuinely non-monotone dependence; it splits further as $\xi-\xiu=\xid+\xi^\circ$ with $\xi^\circ\ge0$ (Proposition~\ref{prop:three}), so that $\xiu$, $\xid$ and $\xi^\circ$ reveal whether a strong functional dependence is increasing, decreasing or neither, which $\xi$ alone cannot do (Example~\ref{ex:tent}). Second, the residual $\xi-\xiu$ is an $L^2$-type measure of departure from stochastic increasingness, a hypothesis that is routinely tested in econometrics \citep{LeeLintonWhang2009,DelgadoEscanciano2012,Seo2018}; Theorem~\ref{thm:proj} shows that the least concave majorant used in these tests is precisely the metric projection in the geometry of $\xi$. Third, weak continuity makes $\xiu$ stable under small perturbations of the distribution and turns the plug-in principle into a consistent estimation method without smoothing or tuning parameters (Theorem~\ref{thm:est}), whereas the plug-in of the empirical distribution always yields $\xi=1$. Fourth, monotonicity in the concordance order makes $\xiu$ compatible with the dependence orderings used to compare dependence structures, for instance in risk management, which $\xi$ is not.

\emph{Related work.} The rearranged dependence measures of \citet{StrothmannDetteSiburg2024} also map a copula to an SI copula, namely by monotone rearrangement of $\partial_1C(\cdot,v)$. Rearrangement preserves $\xi$ and turns any functional dependence into a monotone one, whereas the projection retains only the part of $\xi$ that is explained by non-decreasing functions of $X$; the two constructions are compared in Remark~\ref{rem:rearr}. The identity $\xi=6\,d_1(\cdot,\Pi)^2$ links $\xi$ to the strong metrics on copulas introduced by \citet{Trutschnig2011}; see Section~\ref{sec:prelim}. Finally, Chatterjee's rank correlation and its relation to positive dependence concepts and to concordance measures have attracted considerable interest; see \citet{AnsariRockel2024,AnsariRockel2026,AnsariFuchs2025order,Rockel2026} and the references therein.

The paper is organised as follows. Section~\ref{sec:prelim} collects preliminaries. Section~\ref{sec:proj} studies the projection, Section~\ref{sec:xiu} the monotone part $\xiu$, and Section~\ref{sec:examples} presents examples. Section~\ref{sec:sym} treats the symmetric version, and Section~\ref{sec:general} arbitrary distributions and estimation.

\section{Preliminaries}\label{sec:prelim}

\subsection{Copulas and positive dependence}
Let $\C$ denote the set of bivariate copulas, $\Pi(u,v)=uv$, $M(u,v)=\min(u,v)$ and $W(u,v)=\max(u+v-1,0)$. For $C\in\C$ let $C^\top(u,v)=C(v,u)$. We write $C_1\le C_2$ if $C_1(u,v)\le C_2(u,v)$ for all $(u,v)$ (concordance or pointwise order). Every copula is $1$-Lipschitz in each argument, so for every $v\in[0,1]$ the section $C(\cdot,v)$ is differentiable almost everywhere with $0\le\partial_1C(\cdot,v)\le1$ and
\[
C(u,v)=\int_0^u\partial_1C(s,v)\,ds .
\]
If $(U,V)$ has distribution function $C$, then $\partial_1C(u,v)=\Prob(V\le v\mid U=u)$ for almost every $u$ \citep{DarsowNguyenOlsen1992,Nelsen2006}.

A copula $C$ is \emph{negatively quadrant dependent} (NQD) if $C\le\Pi$ and \emph{positively quadrant dependent} (PQD) if $C\ge\Pi$. It is \emph{stochastically increasing} (SI) if $V$ is stochastically increasing in $U$, i.e., if $\Prob(V\le v\mid U=u)$ is non-increasing in $u$ for every $v$; this is Lehmann's positive regression dependence. Equivalently, $C(\cdot,v)$ is concave for every $v\in[0,1]$ \citep[Chapter~5]{Nelsen2006}. We denote by $\CSI$ the set of SI copulas. It is convex and closed under pointwise convergence. Every SI copula is PQD, since a concave function on $[0,1]$ lies above its chord.

\subsection{The Sobolev scalar product and Chatterjee's rank correlation}
For $A,B$ in the linear span of $\C$ let
\[
\ip{A}{B}_i=\int_0^1\!\!\int_0^1\partial_iA\,\partial_iB\,d\lambda,\quad i=1,2,\qquad \ip{A}{B}=\ip{A}{B}_1+\ip{A}{B}_2 ,
\]
and let $d_i$, $d$ be the corresponding distances, so that $d^2=d_1^2+d_2^2$ and $d_2(A,B)=d_1(A^\top,B^\top)$. Here $\ip{\cdot}{\cdot}$ is the Sobolev scalar product of \citet{SiburgStoimenov2008}. Since every $A\in\C$ satisfies $A(0,v)=0$, $d_1$ is a metric on $\C$. Since $\partial_1A(u,v)=K_A(u,[0,v])$ for the Markov kernel $K_A$ of $A$, the metric $d_1$ coincides with the metric $D_2$ of \citet{Trutschnig2011}; in particular, it generates the same topology on $\C$ as the metrics $D_1$ and $D_\infty$ studied there.

For every $C\in\C$ we have $\ip{C}{\Pi}_1=\int_0^1\!\int_0^1 v\,\partial_1C(u,v)\,du\,dv=\int_0^1v^2\,dv=\tfrac13$ and $\ip{\Pi}{\Pi}_1=\tfrac13$, hence
\[
d_1(C,\Pi)^2=\int\!\!\int(\partial_1C)^2\,d\lambda-\tfrac13=\tfrac16\,\xi(C)
\]
by \eqref{eq:xi-copula}; in terms of Trutschnig's metric, this identity $\xi(C)=6\,D_2(C,\Pi)^2$ has also been noted by \citet{KasperFuchsTrutschnig2021}. In the same way $d_2(C,\Pi)^2=\tfrac16\,\xi(C^\top)$. Together with $\ip{C}{\Pi}=\tfrac23$ \citep[Theorem~12]{SiburgStoimenov2008} this gives \eqref{eq:geom}, where $\omega(C)=(3\|C\|^2-2)^{1/2}$ is the measure of mutual complete dependence of \citet{SiburgStoimenov2010}. Recall that $\xi(C)=0$ if and only if $C=\Pi$, and $\xi(C)=1$ if and only if $\partial_1C\in\{0,1\}$ almost everywhere, i.e., if and only if $V=f(U)$ almost surely for a measurable $f$ \citep{DetteSiburgStoimenov2013}. Analogously, $\omega(C)=1$ if and only if $C$ is invertible with respect to the Markov product \citep{SiburgStoimenov2008}.

\section{The projection onto stochastically increasing copulas}\label{sec:proj}

\subsection{Isotonic regression on the unit interval}
Let $\K\subset L^2(0,1)$ be the set of (equivalence classes of) non-increasing functions. It is a closed convex cone containing all constants; closedness follows because an $L^2$-convergent sequence has an almost everywhere convergent subsequence. Let $\I$ denote the metric projection onto $\K$ in $L^2(0,1)$, i.e., the isotonic regression onto non-increasing functions. For a function $G$ on $[0,1]$, $\lcm G$ denotes its least concave majorant.

\begin{lemma}\label{lem:iso}
Let $g\in L^\infty(0,1)$ and $G(u)=\int_0^ug(s)\,ds$.
\begin{enumerate}
\item[(i)] $\lcm G$ is Lipschitz continuous with $\lcm G(0)=0$ and $\lcm G(1)=G(1)$, and its derivative satisfies $(\lcm G)'=\I g$ almost everywhere. In particular, $\lcm G(u)=\int_0^u\I g(s)\,ds$.
\item[(ii)] $\ip{g-\I g}{\I g}_{L^2}=0$ and $\int_0^1(g-\I g)=0$. Consequently, for every constant $c\in\R$,
\[
\|g-c\|_{L^2}^2=\|\I g-c\|_{L^2}^2+\|g-\I g\|_{L^2}^2 .
\]
\item[(iii)] If $g_1\le g_2$ almost everywhere, then $\I g_1\le\I g_2$ almost everywhere. In particular, $\operatorname{ess\,inf}g\le\I g\le\operatorname{ess\,sup}g$.
\item[(iv)] $\|\I g_1-\I g_2\|_{L^2}\le\|g_1-g_2\|_{L^2}$ and $\sup|\lcm G_1-\lcm G_2|\le\sup|G_1-G_2|$.
\end{enumerate}
\end{lemma}

Parts (i) and (iii) are classical results on isotonic regression \citep{BarlowEtAl1972,RobertsonWrightDykstra1988,GroeneboomJongbloed2014}. We include a short proof to keep the paper self-contained.

\begin{proof}
Since $\K$ is a closed convex cone, $h=\I g$ if and only if
\begin{equation}\label{eq:moreau}
h\in\K,\qquad \ip{g-h}{h}=0,\qquad \ip{g-h}{k}\le0\ \text{ for all }k\in\K .
\end{equation}

(i) Let $L=\|g\|_\infty$. The affine functions $u\mapsto Lu$ and $u\mapsto G(1)+L(1-u)$ are concave majorants of $G$, so $\lcm G(0)=0$, $\lcm G(1)=G(1)$, and the one-sided derivatives of the concave function $\lcm G$ lie in $[-L,L]$. Hence $\lcm G$ is Lipschitz continuous. The set $O=\{\lcm G>G\}$ is open in $(0,1)$, and $\lcm G$ is affine on each of its connected components \citep[Chapter~2]{GroeneboomJongbloed2014}.

Let $h$ be the right derivative of $\lcm G$ on $[0,1)$. It is bounded, non-increasing and right-continuous, so $h\in\K$ and $\int_0^uh=\lcm G(u)$. Let $\nu$ be the finite measure on $(0,1)$ with $\nu((a,b])=h(a)-h(b)$; then $h(u)=h(1-)+\nu((u,1))$. For $f\in L^1(0,1)$ with $F(u)=\int_0^uf$, Fubini's theorem gives
\begin{equation}\label{eq:fubini}
\int_0^1f\,h\,du=h(1-)\,F(1)+\int_{(0,1)}F(s)\,d\nu(s).
\end{equation}
Apply this with $f=g-h$, so that $F=G-\lcm G\le0$, $F(1)=0$ and $F=0$ outside $O$. Since $h$ is constant on each component of $O$, $\nu(O)=0$, hence $\ip{g-h}{h}=0$. For a bounded right-continuous $k\in\K$ with associated measure $\nu_k$, the same identity gives $\ip{g-h}{k}=\int F\,d\nu_k\le0$. For general $k\in\K$, apply this to the truncations $(k\wedge m)\vee(-m)\in\K$ and let $m\to\infty$. By \eqref{eq:moreau}, $h=\I g$.

(ii) The first identity is part of \eqref{eq:moreau}. Taking $k=\pm1$ in \eqref{eq:moreau} gives $\int(g-h)=0$. The Pythagorean identity follows by expanding $\|(g-h)+(h-c)\|^2$.

(iii) Let $h_i=\I g_i$. Since $\K$ is a lattice, $h_1\wedge h_2$ and $h_1\vee h_2$ belong to $\K$. On $\{h_1>h_2\}$ we have
\[
(g_1-h_2)^2+(g_2-h_1)^2-(g_1-h_1)^2-(g_2-h_2)^2=2(g_1-g_2)(h_1-h_2)\le0,
\]
while the expressions agree elsewhere. Hence
\[
\|g_1-h_1\wedge h_2\|^2+\|g_2-h_1\vee h_2\|^2\le\|g_1-h_1\|^2+\|g_2-h_2\|^2 .
\]
Since $h_i$ is the unique minimiser of $\|g_i-k\|$ over $k\in\K$, each summand on the left is at least the corresponding one on the right. Thus equality holds in both, and uniqueness yields $h_1\wedge h_2=h_1$, i.e., $h_1\le h_2$. The second statement follows because constants are fixed by $\I$.

(iv) Metric projections onto closed convex sets are non-expansive. If $\delta=\sup|G_1-G_2|$, then $\lcm G_2+\delta$ is a concave majorant of $G_1$, so $\lcm G_1\le\lcm G_2+\delta$, and vice versa.
\end{proof}

\subsection{The projection}

\begin{definition}\label{def:proj}
For $C\in\C$, the \emph{metric projection} of $C$ onto $\CSI$ with respect to $d_1$ is the copula $\Proj C\in\CSI$ satisfying
\[
d_1(C,\Proj C)=\inf_{D\in\CSI}d_1(C,D)=:d_1(C,\CSI).
\]
\end{definition}

If such a copula exists, it is unique: if $D,D'\in\CSI$ both attain the infimum $\delta$, then $\frac12(D+D')\in\CSI$ by convexity, and the parallelogram law gives
\[
d_1(D,D')^2=2\,d_1(C,D)^2+2\,d_1(C,D')^2-4\,d_1\bigl(C,\tfrac12(D+D')\bigr)^2\le4\delta^2-4\delta^2=0,
\]
so $D=D'$. Existence is part of the following theorem.

\begin{theorem}\label{thm:proj}
For every $C\in\C$ the metric projection $\Proj C$ exists, and the following hold.
\begin{enumerate}
\item[(a)] $\partial_1\Proj C(\cdot,v)=\I\bigl(\partial_1C(\cdot,v)\bigr)$ almost everywhere, for every $v\in[0,1]$.
\item[(b)] $C\le\Proj C$ and $\Pi\le\Proj C$. If $D\in\CSI$ and $D\ge C$, then $D\ge\Proj C$. Thus $\Proj C$ is the smallest SI copula dominating $C$. Moreover, $\Proj C=C$ if and only if $C\in\CSI$.
\item[(c)] For every $v\in[0,1]$, the section $\Proj C(\cdot,v)$ is the least concave majorant of $C(\cdot,v)$.
\item[(d)] $d_1(C,\Pi)^2=d_1(\Proj C,\Pi)^2+d_1(C,\Proj C)^2$.
\item[(e)] If $C_1\le C_2$, then $\Proj C_1\le\Proj C_2$. Moreover,
\[
\|\Proj C_1-\Proj C_2\|_\infty\le\|C_1-C_2\|_\infty\quad\text{and}\quad d_1(\Proj C_1,\Proj C_2)\le d_1(C_1,C_2).
\]
\end{enumerate}
\end{theorem}

\begin{proof}
For $v\in[0,1]$ put $g_v=\partial_1C(\cdot,v)$ and $G_v=C(\cdot,v)=\int_0^\cdot g_v$, and define $Q(u,v)=\lcm G_v(u)$. We first show that $Q$ is an SI copula with $\partial_1Q(\cdot,v)=\I g_v$, and then that $Q$ attains the infimum in Definition~\ref{def:proj}.

By Lemma~\ref{lem:iso}(i), $Q(\cdot,v)$ is Lipschitz continuous with $Q(0,v)=0$, $Q(1,v)=v$ and derivative $\I g_v$. For $v=0$ we have $G_0\equiv0$, and for $v=1$ the section $G_1(u)=u$ is concave; hence $Q(u,0)=0$ and $Q(u,1)=u$. Let $0\le v_1<v_2\le1$. Since $C$ is $2$-increasing, $G_{v_2}-G_{v_1}$ is non-decreasing, so $g_{v_1}\le g_{v_2}$ almost everywhere. By Lemma~\ref{lem:iso}(iii), $\I g_{v_1}\le\I g_{v_2}$, so
\[
u\mapsto Q(u,v_2)-Q(u,v_1)=\int_0^u\bigl(\I g_{v_2}-\I g_{v_1}\bigr)
\]
is non-decreasing. This is the $2$-increasing property. Thus $Q$ is a copula, and it is SI because its vertical sections are concave.

Now let $D\in\CSI$. For every $v$, $\partial_1D(\cdot,v)\in\K$, so $\|g_v-\partial_1D(\cdot,v)\|_{L^2}\ge\|g_v-\I g_v\|_{L^2}=\|g_v-\partial_1Q(\cdot,v)\|_{L^2}$. Integration over $v$ gives $d_1(C,D)\ge d_1(C,Q)$. Hence the infimum in Definition~\ref{def:proj} is attained at $Q\in\CSI$, so that $\Proj C$ exists and $\Proj C=Q$. This proves (c), and (a) follows from the first step.

(b) $C\le\Proj C$ because $\Proj C(\cdot,v)$ is a majorant of $G_v$. The concave function $\Proj C(\cdot,v)$ takes the values $0$ and $v$ at $0$ and $1$, hence lies above $u\mapsto uv$. If $D\in\CSI$ and $D\ge C$, then $D(\cdot,v)$ is a concave majorant of $G_v$, hence $D(\cdot,v)\ge\lcm G_v=\Proj C(\cdot,v)$. Finally, $\Proj C=C$ if and only if $C\in\CSI$, by Definition~\ref{def:proj} and the uniqueness of the metric projection.

(d) Apply Lemma~\ref{lem:iso}(ii) to $g_v$ with $c=v=\partial_1\Pi(\cdot,v)$ and integrate over $v$.

(e) If $C_1\le C_2$, then $\Proj C_2(\cdot,v)$ is a concave majorant of $C_1(\cdot,v)$ by (c), so $\Proj C_1\le\Proj C_2$. The two estimates follow from (c) and Lemma~\ref{lem:iso}(iv), the second after integration over $v$.
\end{proof}

\begin{remark}\label{rem:comp}
For a checkerboard copula, $\partial_1C(\cdot,v)$ is a step function, and $\I$ reduces to the pool-adjacent-violators algorithm with equal weights. The least concave majorant of copula sections is the central object of the tests of stochastic monotonicity of \citet{DelgadoEscanciano2012} and \citet{Seo2018}. Theorem~\ref{thm:proj}(c) shows that it coincides with the metric projection in the geometry of $\xi$.
\end{remark}

\begin{remark}\label{rem:rearr}
\citet{StrothmannDetteSiburg2024} assign to $C$ the SI copula $C^{\mathrm{re}}$ whose derivative $\partial_1C^{\mathrm{re}}(\cdot,v)$ is the non-increasing rearrangement of $\partial_1C(\cdot,v)$. By the Hardy--Littlewood inequality,
\[
\int_0^u\partial_1C(s,v)\,ds\le\int_0^u\partial_1C^{\mathrm{re}}(s,v)\,ds\quad\text{for all }u,v\in[0,1],
\]
i.e., $C\le C^{\mathrm{re}}$. Theorem~\ref{thm:proj}(b) therefore gives $C\le\Proj C\le C^{\mathrm{re}}$.
Rearrangement preserves the $L^2$-norm of every section, so $\xi(C^{\mathrm{re}})=\xi(C)$: it reorders the values of $X$ until the entire functional dependence becomes monotone. The projection does not reorder; it retains exactly the part of $\xi$ that is explained by non-decreasing functions of $X$ and discards $6\,d_1(C,\CSI)^2$.
\end{remark}

\section{The monotone part of Chatterjee's rank correlation}\label{sec:xiu}

\begin{definition}\label{def:xiu}
For $C\in\C$ let
\[
\xiu(C)=\xi(\Proj C)=6\,d_1(\Proj C,\Pi)^2,\qquad \xid(C)=\xiu(C^-),
\]
where $C^-(u,v)=v-C(1-u,v)$ is the copula of $(1-U,V)$ for $(U,V)\sim C$.
\end{definition}

Throughout, the \emph{monotone part} $\xiu$ refers to increasing dependence; decreasing dependence is captured by $\xid$. In particular, $\xi-\xiu$ is not the non-monotone part of $\xi$: for the countermonotonicity copula, $\xi(W)-\xiu(W)=\xid(W)=1$. A genuinely non-monotone part is identified in Proposition~\ref{prop:three}.

We need an elementary majorisation inequality.

\begin{lemma}\label{lem:hlp}
Let $h_1,h_2\colon[0,1]\to[0,1]$ be non-increasing with $\int_0^uh_1\le\int_0^uh_2$ for all $u\in[0,1]$ and equality for $u=1$. Then $\int_0^1h_1^2\le\int_0^1h_2^2$.
\end{lemma}

\begin{proof}
Put $\Delta(u)=\int_0^u(h_2-h_1)\ge0$, so that $\Delta(0)=\Delta(1)=0$. Then
\[
\int h_2^2-\int h_1^2=\int(h_2-h_1)^2+2\int_0^1h_1\,d\Delta\ \ge\ 2\int_0^1h_1\,d\Delta=-2\int_0^1\Delta\,dh_1\ \ge0 .
\]
The equality follows from integration by parts for Stieltjes integrals, since $h_1$ has bounded variation and $\Delta$ is continuous and vanishes at the endpoints. The last inequality holds because $\Delta\ge0$ and, $h_1$ being non-increasing, the Lebesgue--Stieltjes measure $dh_1$ is non-positive.
\end{proof}

Lemma~\ref{lem:hlp} is the case $\varphi(x)=x^2$ of the Hardy--Littlewood--P\'olya inequality \citep{HardyLittlewoodPolya1952}. It implies that $\xi$ is increasing in the concordance order on $\CSI$, as observed by \citet{AnsariRockel2026}.

\emph{Polarity of stochastic increasingness and negative quadrant dependence.} The decomposition of $\xi$ has a second, dual reading, which rests on the Hilbert space structure. Let $\mathcal H$ be the space of all functions $A(u,v)=\int_0^ua(s,v)\,ds$ with $a\in L^2([0,1]^2)$, equipped with $\ip{\cdot}{\cdot}_1$; it is isometric to $L^2([0,1]^2)$ and contains $C-\Pi$ for every $C\in\C$. Let
\[
\mathfrak K=\bigl\{A\in\mathcal H:\ A(\cdot,v)\ \text{is concave for almost every }v\bigr\},
\]
a closed convex cone, and let $\mathfrak K^\circ=\{B\in\mathcal H:\ip{B}{A}_1\le0\text{ for all }A\in\mathfrak K\}$ be its polar cone. By the theorem of \citet{Moreau1962}, every $x\in\mathcal H$ decomposes uniquely as $x=P_{\mathfrak K}x+P_{\mathfrak K^\circ}x$ with $\ip{P_{\mathfrak K}x}{P_{\mathfrak K^\circ}x}_1=0$, where $P$ denotes the metric projection, and $\|P_{\mathfrak K}x\|_1=\operatorname{dist}_1(x,\mathfrak K^\circ)$, $\|P_{\mathfrak K^\circ}x\|_1=\operatorname{dist}_1(x,\mathfrak K)$.

\begin{theorem}\label{thm:polar}
The following hold.
\begin{enumerate}
\item[(a)] $\mathfrak K^\circ=\bigl\{B\in\mathcal H:\ B(\cdot,v)\le0\text{ and }B(1,v)=0\text{ for almost every }v\bigr\}$.
\item[(b)] $\CSI=\C\cap(\Pi+\mathfrak K)$ and $\C_{\mathrm{NQD}}:=\{C\in\C:C\le\Pi\}=\C\cap(\Pi+\mathfrak K^\circ)$.
\item[(c)] For every $C\in\C$, $P_{\mathfrak K}(C-\Pi)=\Proj C-\Pi$ and $P_{\mathfrak K^\circ}(C-\Pi)=C-\Proj C$. Thus
\[
C-\Pi=(\Proj C-\Pi)+(C-\Proj C)
\]
is the Moreau decomposition of $C-\Pi$, and in particular $\ip{\Proj C-\Pi}{C-\Proj C}_1=0$.
\item[(d)] For every $C\in\C$,
\begin{align*}
\xiu(C)&=6\operatorname{dist}_1(C-\Pi,\mathfrak K^\circ)^2\le6\,d_1(C,\C_{\mathrm{NQD}})^2,\\
\xi(C)-\xiu(C)&=6\operatorname{dist}_1(C-\Pi,\mathfrak K)^2=6\,d_1(C,\CSI)^2 .
\end{align*}
\end{enumerate}
\end{theorem}

\begin{proof}
(a) We first identify the polar cone $\K^\circ$ of $\K$ in $L^2(0,1)$: for $w\in L^2(0,1)$ with $W(u)=\int_0^uw$, we claim that $w\in\K^\circ$ if and only if $W\le0$ and $W(1)=0$. If $W\le0$ and $W(1)=0$, then \eqref{eq:fubini} gives $\ip{w}{k}=\int W\,d\nu_k\le0$ for every bounded right-continuous $k\in\K$, and truncation extends this to all $k\in\K$. Conversely, testing with $k=\pm1$ and $k=\1_{[0,u)}$ gives $W(1)=0$ and $W(u)\le0$. Now let $B\in\mathcal H$ with $B(u,v)=\int_0^ub(s,v)\,ds$. If $b(\cdot,v)\in\K^\circ$ for almost every $v$, then $\ip{B}{A}_1=\int_0^1\ip{b(\cdot,v)}{a(\cdot,v)}\,dv\le0$ for all $A\in\mathfrak K$. If not, then there are a rational $u$, a sign $\sigma\in\{\pm1\}$ and a set $V\subset[0,1]$ of positive measure such that $B(u,v)>0$ for all $v\in V$, or such that $\sigma B(1,v)>0$ for all $v\in V$. The function $A$ with $\partial_1A(s,v)=\1_{[0,u)}(s)\1_V(v)$, respectively $\partial_1A(s,v)=\sigma\1_V(v)$, belongs to $\mathfrak K$ and satisfies $\ip{B}{A}_1>0$, so $B\notin\mathfrak K^\circ$.

(b) A copula $D$ is SI if and only if every section $D(\cdot,v)$ is concave, i.e., if and only if $D-\Pi\in\mathfrak K$. Moreover, $D-\Pi$ vanishes at $u=1$ for every copula $D$, so (a) shows that $D-\Pi\in\mathfrak K^\circ$ if and only if $D\le\Pi$.

(c) For $A\in\mathfrak K$ we have $\|C-\Pi-A\|_1^2=\int_0^1\|g_v-v-a(\cdot,v)\|_{L^2}^2\,dv\ge\int_0^1\|g_v-v-\I(g_v-v)\|_{L^2}^2\,dv$, where $g_v=\partial_1C(\cdot,v)$. Since $\I(g_v-v)=\I g_v-v=\partial_1(\Proj C-\Pi)(\cdot,v)$ by Theorem~\ref{thm:proj}(a), equality is attained at $A=\Proj C-\Pi\in\mathfrak K$. Hence $P_{\mathfrak K}(C-\Pi)=\Proj C-\Pi$, and Moreau's theorem gives $P_{\mathfrak K^\circ}(C-\Pi)=C-\Proj C$.

(d) By (c), \eqref{eq:geom} and Moreau's theorem, $\xiu(C)=6\|\Proj C-\Pi\|_1^2=6\|P_{\mathfrak K}(C-\Pi)\|_1^2=6\operatorname{dist}_1(C-\Pi,\mathfrak K^\circ)^2$. For every $D\in\C_{\mathrm{NQD}}$ we have $D-\Pi\in\mathfrak K^\circ$ by (b), so $\operatorname{dist}_1(C-\Pi,\mathfrak K^\circ)\le d_1(C,D)$. The second identity follows in the same way from $\|C-\Proj C\|_1=d_1(C,\CSI)$.
\end{proof}

\begin{remark}\label{rem:polar}
Theorem~\ref{thm:polar} shows that, in the tangent space at $\Pi$ of the geometry in which $\xi$ is a squared distance, stochastic increasingness and negative quadrant dependence correspond to mutually polar cones. The monotone part $\xiu$ measures the distance to the (relaxed) class of negatively quadrant dependent directions, and the remainder $\xi-\xiu$ measures the distance to stochastic increasingness; it contains decreasing as well as genuinely non-monotone dependence (Proposition~\ref{prop:three}). The two projections behave differently: $\Pi+P_{\mathfrak K}(C-\Pi)=\Proj C$ is always a copula (Theorem~\ref{thm:proj}), whereas $\Pi+P_{\mathfrak K^\circ}(C-\Pi)=\Pi+C-\Proj C$ need not be one. For the copula $C$ of Example~\ref{ex:tent}, for instance, $\partial_1(\Pi+C-\Proj C)(u,v)=v(v-1)/(1+v)<0$ for $u<\frac{1-v}2$. In this case the inequality in Theorem~\ref{thm:polar}(d) is strict: otherwise a minimising sequence $D_n\in\C_{\mathrm{NQD}}$ would satisfy $D_n-\Pi\to P_{\mathfrak K^\circ}(C-\Pi)$ in $\mathcal H$, as every minimising sequence in a closed convex set does, and completeness of $(\C,d_1)$ \citep{Trutschnig2011} would make $\Pi+C-\Proj C$ a copula. If $\Pi+C-\Proj C$ is a copula, it is the closest NQD copula to $C$, and equality holds. Applying Theorem~\ref{thm:polar} to $C^-$ yields the analogous polarity of stochastically decreasing and positively quadrant dependent copulas, with $\xid$ in place of $\xiu$:
\begin{center}
\begin{tabular}{lll}
\toprule
Cone of directions & Polar cone & Monotone part\\
\midrule
SI (concave sections) & NQD ($C\le\Pi$) & $\xiu$\\
SD (convex sections) & PQD ($C\ge\Pi$) & $\xid$\\
\bottomrule
\end{tabular}
\end{center}
\end{remark}

\begin{theorem}\label{thm:xiu}
For all $C,C_1,C_2,C_n\in\C$ the following hold.
\begin{enumerate}
\item[(a)] \emph{Decomposition.} $\xi(C)=\xiu(C)+6\,d_1(C,\CSI)^2$; in particular, $0\le\xiu(C)\le\xi(C)\le1$.
\item[(b)] $\xiu(C)=\xi(C)$ if and only if $C\in\CSI$.
\item[(c)] $\xiu(C)=0$ if and only if $C\le\Pi$, i.e., $C$ is NQD.
\item[(d)] $\xiu(C)=1$ if and only if $C=M$.
\item[(e)] \emph{Concordance monotonicity.} If $C_1\le C_2$, then $\xiu(C_1)\le\xiu(C_2)$.
\item[(f)] \emph{Continuity.} $\bigl|\sqrt{\xiu(C_1)}-\sqrt{\xiu(C_2)}\bigr|\le\sqrt6\,d_1(C_1,C_2)$. Moreover, $\xiu$ is weakly continuous: if $C_n\to C$ uniformly (equivalently, if the associated probability measures converge weakly), then $\xiu(C_n)\to\xiu(C)$.
\item[(g)] \emph{Spearman bound.} $\xiu(C)\ge\tfrac12\bigl(\rho_S(C)_+\bigr)^2$, where $\rho_S(C)=12\int C\,d\lambda-3$ is Spearman's rho. The constant $\frac12$ is best possible, and the inequality is strict whenever $\rho_S(C)>0$.
\end{enumerate}
\end{theorem}

\begin{proof}
(a)--(c) By \eqref{eq:geom} and Theorem~\ref{thm:polar}(c),(d),
\[
\xi(C)=6\|C-\Pi\|_1^2=6\|P_{\mathfrak K}(C-\Pi)\|_1^2+6\|P_{\mathfrak K^\circ}(C-\Pi)\|_1^2=\xiu(C)+6\,d_1(C,\CSI)^2,
\]
which is (a). Moreover, $\xiu(C)=\xi(C)$ if and only if $P_{\mathfrak K^\circ}(C-\Pi)=0$, i.e., $C-\Pi\in\mathfrak K$, and $\xiu(C)=0$ if and only if $P_{\mathfrak K}(C-\Pi)=0$, i.e., $C-\Pi\in\mathfrak K^\circ$. By Theorem~\ref{thm:polar}(b), this proves (b) and (c).

(d) $\xiu(M)=\xi(M)=1$. Conversely, $\xiu(C)=1$ forces $\xi(C)=1$ and $C\in\CSI$ by (a) and (b). Then $\partial_1C\in\{0,1\}$ almost everywhere. For almost every $v$, $\partial_1C(\cdot,v)$ is therefore a non-increasing $\{0,1\}$-valued function with integral $v$, i.e., it equals $\1_{[0,v)}$ almost everywhere. Hence $C(u,v)=\min(u,v)$ for almost every $v$, and $C=M$ by continuity.

(e) By Theorem~\ref{thm:proj}(e), $\Proj C_1\le\Proj C_2$. For fixed $v$, let $h_i$ be the right derivative of the concave function $\Proj C_i(\cdot,v)$; then $h_i=\partial_1\Proj C_i(\cdot,v)$ almost everywhere, and $h_i$ is non-increasing with values in $[0,1]$ (Lemma~\ref{lem:iso}(iii)), and $\int_0^uh_1=\Proj C_1(u,v)\le\Proj C_2(u,v)=\int_0^uh_2$ with equality at $u=1$. Lemma~\ref{lem:hlp} and integration over $v$ give $\int(\partial_1\Proj C_1)^2\le\int(\partial_1\Proj C_2)^2$, and \eqref{eq:xi-copula} yields the claim.

(f) Since $\sqrt{\xiu(C)}=\sqrt6\,\|P_{\mathfrak K}(C-\Pi)\|_1$ and metric projections onto closed convex sets are non-expansive, $\bigl|\sqrt{\xiu(C_1)}-\sqrt{\xiu(C_2)}\bigr|\le\sqrt6\,\|P_{\mathfrak K}(C_1-\Pi)-P_{\mathfrak K}(C_2-\Pi)\|_1\le\sqrt6\,d_1(C_1,C_2)$. For weak continuity, note that $\Proj C_n\to\Proj C$ uniformly by Theorem~\ref{thm:proj}(e). Fix $v$. The concave functions $\varphi_n=\Proj C_n(\cdot,v)$ converge uniformly to the concave function $\varphi=\Proj C(\cdot,v)$. At every $u\in(0,1)$ where $\varphi$ is differentiable, the right derivatives of $\varphi_n$ converge to $\varphi'(u)$ \citep[Theorem~24.5]{Rockafellar1970}. Hence $\partial_1\Proj C_n(u,v)\to\partial_1\Proj C(u,v)$ for almost every $u$ and every $v$, and thus for $\lambda^2$-almost every $(u,v)$ by Fubini's theorem. All these functions take values in $[0,1]$, so dominated convergence gives $\int(\partial_1\Proj C_n)^2\to\int(\partial_1\Proj C)^2$.

(g) For $D\in\C$, Fubini's theorem and $\int_0^1(\partial_1D(s,v)-v)\,ds=0$ give
\[
\int_0^1\bigl(D(u,v)-uv\bigr)\,du=\int_0^1\bigl(\partial_1D(s,v)-v\bigr)(1-s)\,ds=\int_0^1\bigl(\partial_1D(s,v)-v\bigr)\bigl(\tfrac12-s\bigr)\,ds .
\]
Integrating over $v$ and applying the Cauchy--Schwarz inequality, we obtain
\[
\rho_S(D)\le12\,d_1(D,\Pi)\,\Bigl(\int_0^1\!\!\int_0^1(\tfrac12-s)^2\,ds\,dv\Bigr)^{1/2}=12\sqrt{\xi(D)/6}\,\sqrt{1/12}=\sqrt{2\,\xi(D)} .
\]
Since $C\le\Proj C$, it follows that $\rho_S(C)\le\rho_S(\Proj C)\le\sqrt{2\,\xiu(C)}$. Now let $\rho_S(C)>0$ and suppose that equality holds. Then $\rho_S(C)=\rho_S(D)$ for $D=\Proj C$, and equality holds in the Cauchy--Schwarz inequality, i.e., $\partial_1D(s,v)-v=\lambda(\frac12-s)$ almost everywhere for some $\lambda\ge0$. Since $\partial_1D\ge0$, this gives $\lambda\le2v$ for almost every $v\in(0,1)$, hence $\lambda=0$ and $\rho_S(D)=12\lambda\int_0^1\!\int_0^1(\frac12-s)^2\,ds\,dv=0$, a contradiction. Thus the inequality is strict if $\rho_S(C)>0$. For optimality of $\frac12$, let $0<c\le1$, $c(v)=\min\{c,2v,2(1-v)\}$ and $D_c(u,v)=uv+\frac12c(v)\,u(1-u)$. Then $\partial_1D_c(u,v)=v+c(v)(\frac12-u)$ takes values in $[0,1]$, because $c(v)\le\min\{2v,2(1-v)\}$, and it is non-increasing in $u$. Since $c$ is Lipschitz continuous with $|c'|\le2$, we have $\partial_v\partial_1D_c(u,v)=1+c'(v)(\frac12-u)\ge0$, so $\partial_1D_c$ is non-decreasing in $v$. Together with $D_c(0,v)=0$, $D_c(1,v)=v$, $D_c(u,0)=0$ and $D_c(u,1)=u$, which hold because $c(0)=c(1)=0$, this shows that $D_c$ is an SI copula. A direct computation gives $\rho_S(D_c)=\int_0^1c(v)\,dv=c(1-\frac c2)$ and $\xiu(D_c)=\xi(D_c)=\frac12\int_0^1c(v)^2\,dv=\frac{c^2(3-2c)}6$, so that
\[
\frac{\xiu(D_c)}{\rho_S(D_c)^2}=\frac{2(3-2c)}{3(2-c)^2}\ \longrightarrow\ \frac12\qquad(c\downarrow0).
\]
\end{proof}

\begin{remark}
Properties (c)--(f) distinguish $\xiu$ from $\xi$. The comonotonicity and countermonotonicity copulas satisfy $\xi(M)=\xi(W)=1$, whereas $\xiu(M)=1$ and $\xiu(W)=0$. Weak continuity fails for $\xi$ \citep{BucherDette2025}: Example~\ref{ex:saw} provides copulas $C_k\to\Pi$ uniformly with $\xi(C_k)=1$ for all $k$, whereas $\xiu(C_k)=1/k\to0=\xiu(\Pi)$. In (g), the constant $\frac12$ is optimal but not attained. For comparison, the Farlie--Gumbel--Morgenstern copulas with $\theta>0$ give the ratio $\xiu/\rho_S^2=\frac35$. The exact region of $(\xi,\rho_S)$ over all copulas has been determined by \citet{AnsariRockel2026}.
\end{remark}

Since $\partial_1C^-(u,v)=\partial_1C(1-u,v)$, we have $\xi(C^-)=\xi(C)$. Moreover, $C^-\le\Pi$ if and only if $C\ge\Pi$, $C^-=M$ if and only if $C=W$, and $C\mapsto C^-$ reverses the concordance order and preserves uniform convergence. Hence Theorem~\ref{thm:xiu} applied to $C$ and $C^-$ gives the following result.

\begin{corollary}\label{cor:max}
The functional $\kappa(C)=\max\{\xiu(C),\xid(C)\}$ satisfies $0\le\kappa\le\xi$; moreover, $\kappa(C)=0$ if and only if $C=\Pi$, and $\kappa(C)=1$ if and only if $C\in\{M,W\}$. It is continuous with respect to uniform convergence. Furthermore, $\xid$ is decreasing in the concordance order, and $\xid(C)=0$ if and only if $C$ is PQD.
\end{corollary}

The two directed parts do not exhaust $\xi$, but the polarity of Theorem~\ref{thm:polar} shows that their sum never exceeds $\xi$. Let $R$ be the map $(RA)(u,v)=A(1,v)-A(1-u,v)$ on $\mathcal H$. Since $\partial_1(RA)(u,v)=\partial_1A(1-u,v)$, $R$ is a linear isometric involution, $R(C-\Pi)=C^--\Pi$, and $R\mathfrak K=\{A\in\mathcal H:A(\cdot,v)\text{ convex for almost every }v\}$.

\begin{proposition}\label{prop:three}
For every $C\in\C$,
\[
\xi(C)=\xiu(C)+\xid(C)+\xi^\circ(C)\qquad\text{with}\qquad \xi^\circ(C)\ge0 .
\]
More precisely, with $x=C-\Pi$, $p_1=P_{\mathfrak K}x$ and $p_2=P_{R\mathfrak K}x$,
\[
\xi^\circ(C)=6\,\|x-p_1-p_2\|_1^2+12\,|\ip{p_1}{p_2}_1| .
\]
In particular, $\xi^\circ(C)=0$ whenever every section $C(\cdot,v)$ is concave or convex; this includes all SI and all stochastically decreasing copulas.
\end{proposition}

\begin{proof}
Since $R$ is a linear isometric involution, $P_{R\mathfrak K}=RP_{\mathfrak K}R$, and hence $\xid(C)=\xiu(C^-)=6\|P_{\mathfrak K}Rx\|_1^2=6\|p_2\|_1^2$; likewise $\xiu(C)=6\|p_1\|_1^2$. By Theorem~\ref{thm:proj}(b), $p_1=\Proj C-\Pi\ge0$, and $p_1(1,v)=0$. By Theorem~\ref{thm:polar}(a), $(R\mathfrak K)^\circ=R(\mathfrak K^\circ)=\{B\in\mathcal H:B(\cdot,v)\ge0,\ B(1,v)=0\text{ for almost every }v\}$, so $p_1\in(R\mathfrak K)^\circ$ and therefore $\ip{p_1}{p_2}_1\le0$. Moreau's theorem gives $\ip{x}{p_i}_1=\|p_i\|_1^2$, hence
\[
\|x-p_1-p_2\|_1^2=\|x\|_1^2-\|p_1\|_1^2-\|p_2\|_1^2+2\ip{p_1}{p_2}_1 .
\]
Multiplying by $6$ and using $\xi(C)=6\|x\|_1^2$ yields the formula for $\xi^\circ(C)$. Both projections act section by section. If $C(\cdot,v)$ is concave, then $\partial_1x(\cdot,v)$ is non-increasing with mean zero, so $p_1(\cdot,v)=x(\cdot,v)$; moreover, the projection of $\partial_1x(\cdot,v)$ onto the non-decreasing functions is constant, equal to its mean zero, by Lemma~\ref{lem:iso}(i) applied to the reflected function (the least concave majorant of a convex function is its chord), so $p_2(\cdot,v)=0$. The convex case is symmetric. In both cases the integrands of $\|x-p_1-p_2\|_1^2$ and $\ip{p_1}{p_2}_1$ vanish at $v$.
\end{proof}

The remainder $\xi^\circ$ is a natural candidate for the genuinely non-monotone part of $\xi$, i.e., the part explained neither by an increasing nor by a decreasing relationship. Unlike $\xiu$ and $\xid$, however, it inherits the lack of weak continuity from $\xi$: for the sawtooth shuffles of Example~\ref{ex:saw}, $\xi^\circ(C_k)=1-1/k\to1$, whereas $\xi^\circ(\Pi)=0$. For the copula of $(U,|2U-1|)$ in Example~\ref{ex:tent}, $\xi^\circ=17-24\log2\approx0.364$, and for $\frac12M+\frac12W$, $\xi^\circ=\frac{17}4-6\log2\approx0.091$.

\section{Examples}\label{sec:examples}

\begin{example}[Farlie--Gumbel--Morgenstern]
For $C_\theta(u,v)=uv\bigl(1+\theta(1-u)(1-v)\bigr)$, $\theta\in[-1,1]$, we have $\partial_1C_\theta(u,v)=v+\theta v(1-v)(1-2u)$. Hence $C_\theta$ is SI for $\theta\ge0$ and NQD for $\theta\le0$, and
\[
\xi(C_\theta)=\frac{\theta^2}{15},\qquad \xiu(C_\theta)=\frac{(\theta_+)^2}{15},\qquad \xid(C_\theta)=\frac{(\theta_-)^2}{15}.
\]
\end{example}

\begin{example}[Gaussian copulas]
Let $C_\rho$ be the Gaussian copula with correlation $\rho\in(-1,1)$. Then $\int_0^1(\partial_1C_\rho(u,v))^2\,du=C_{\rho^2}(v,v)$, since, given $U$, two conditionally independent copies of $V$ have normal scores with correlation $\rho^2$. Moreover, if $Z_1,Z_2,Z_3$ are standard normal, $(Z_1,Z_2)$ has correlation $r$ and $Z_3$ is independent of $(Z_1,Z_2)$, then
\[
\int_0^1C_r(t,t)\,dt=\Prob(Z_1\le Z_3,\,Z_2\le Z_3)=\frac14+\frac{1}{2\pi}\arcsin\frac{1+r}{2},
\]
because $(Z_1-Z_3,Z_2-Z_3)$ is centred normal with correlation $(1+r)/2$. Therefore $\xi(C_\rho)=\frac3\pi\arcsin\frac{1+\rho^2}{2}-\frac12$. Since $C_\rho$ is SI for $\rho\ge0$ and NQD for $\rho\le0$ \citep{Lehmann1966},
\[
\xiu(C_\rho)=\Bigl(\frac3\pi\arcsin\frac{1+\rho^2}{2}-\frac12\Bigr)\1\{\rho\ge0\}.
\]
\end{example}

\begin{example}[Mixtures of the Fr\'echet--Hoeffding bounds]\label{ex:mix}
Let $C_\alpha=\alpha M+(1-\alpha)W$, $\alpha\in[0,1]$. By \citet[Remark~19]{SiburgStoimenov2008}, $\xi(C_\alpha)=\omega(C_\alpha)^2=1-3\alpha(1-\alpha)$, which is not monotone in $\alpha$, although $C_\alpha$ increases in the concordance order. For $v\le\frac12$, the section $C_\alpha(\cdot,v)$ has slope $\alpha$ on $[0,v]$, slope $0$ on $[v,1-v]$ and slope $1-\alpha$ on $[1-v,1]$. Its least concave majorant has slopes $\alpha$ on $[0,v]$ and $v(1-\alpha)/(1-v)$ on $[v,1]$ if $v\le\alpha$, and it is linear otherwise. Hence $\int_0^1(\partial_1\Proj C_\alpha-v)^2\,du=(\alpha-v)_+^2\,v/(1-v)$. For $v>\frac12$, an analogous computation shows that this integral equals $(\alpha-(1-v))_+^2\,(1-v)/v$, i.e., the same expression with $v$ replaced by $1-v$. Therefore
\[
\xiu(C_\alpha)=12\int_0^{\min(\alpha,1/2)}(\alpha-v)^2\frac{v}{1-v}\,dv,
\]
which evaluates to
\[
\xiu(C_\alpha)=
\begin{cases}
-4\alpha^3+18\alpha^2-12\alpha-12(1-\alpha)^2\log(1-\alpha), & \alpha\le\frac12,\\[2pt]
-6\alpha^2+15\alpha-8+12(1-\alpha)^2\log2, & \alpha\ge\frac12 .
\end{cases}
\]
By symmetry, $\xid(C_\alpha)=\xiu(C_{1-\alpha})$. Figure~\ref{fig:mix} shows the three functions. In accordance with Theorem~\ref{thm:xiu}(e), $\xiu(C_\alpha)$ is increasing in $\alpha$.
\end{example}

\begin{figure}[t]
\centering
\includegraphics[width=0.62\textwidth]{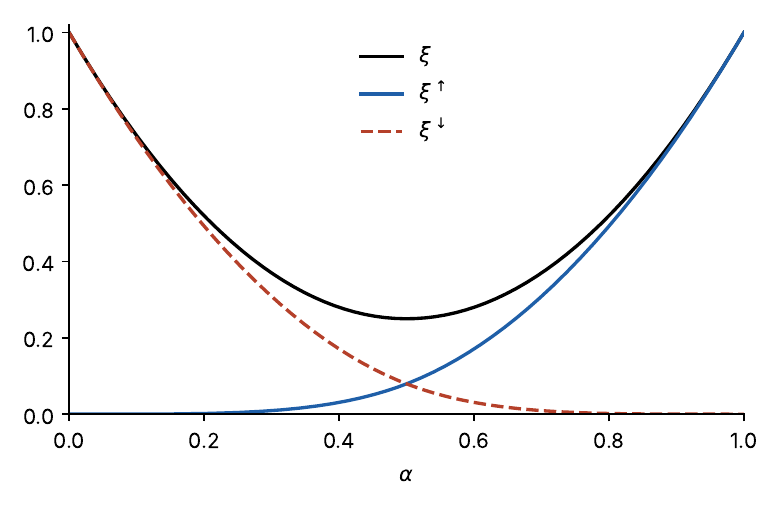}
\caption{$\xi$, $\xiu$ and $\xid$ for the mixtures $\alpha M+(1-\alpha)W$ of Example~\ref{ex:mix}.}
\label{fig:mix}
\end{figure}

\begin{example}[Sawtooth shuffles]\label{ex:saw}
For $k\in\mathbb N$ let $C_k$ be the copula of $(U,kU\bmod1)$ with $U$ uniformly distributed. Then $\xi(C_k)=1$, while $C_k\to\Pi$ uniformly. For fixed $v$, $\partial_1C_k(\cdot,v)$ is the indicator function of $\bigcup_{i=0}^{k-1}[\tfrac ik,\tfrac{i+v}k]$. The least concave majorant of $C_k(\cdot,v)$ passes through $(0,0)$, the points $(\frac{i+v}{k},\frac{(i+1)v}{k})$, $i=0,\dots,k-1$, and $(1,v)$. Its derivative equals $1$ on $[0,\frac vk)$, $v$ on $[\frac vk,\frac{k-1+v}{k})$ and $0$ on $[\frac{k-1+v}k,1]$. Hence $\int_0^1(\partial_1\Proj C_k-v)^2\,du=v(1-v)/k$ and
\[
\xiu(C_k)=\frac1k\longrightarrow0=\xiu(\Pi).
\]
Moreover, $C_k(i/k,v)=iv/k$ and $C_k(\cdot,v)$ is piecewise linear with slopes in $\{0,1\}$, so $C_k\ge\Pi$ and $\xid(C_k)=0$ by Corollary~\ref{cor:max}.
\end{example}

\begin{example}[A non-monotone functional relationship]\label{ex:tent}
Let $C$ be the copula of $(U,|2U-1|)$, i.e., of $(X,(X-\frac12)^2)$ for uniformly distributed $X$. Here $\xi(C)=1$. The section $C(\cdot,v)$ vanishes on $[0,\frac{1-v}2]$, has slope $1$ on $[\frac{1-v}2,\frac{1+v}2]$ and is constant afterwards. Its least concave majorant has slope $2v/(1+v)$ on $[0,\frac{1+v}2]$ and slope $0$ afterwards. This gives $\int_0^1(\partial_1\Proj C-v)^2\,du=v^2(1-v)/(1+v)$ and
\[
\xiu(C)=6\int_0^1\frac{v^2(1-v)}{1+v}\,dv=12\log2-8\approx0.318 .
\]
By the symmetry $u\mapsto1-u$, also $\xid(C)=12\log2-8$. Thus a U-shaped relationship, which is neither PQD nor NQD, has non-trivial increasing and decreasing parts. Here $\rho_S(C)=0$, so concordance measures do not detect the dependence at all, whereas $\xi(C)=1$ does not reveal its shape. The value of $\xiu(C)$ does not mean that the relationship is ``partly monotone'' as a function. It is the share of the variation of the indicators $\1\{Y\ge t\}$ that non-decreasing functions of $X$ can explain (Remark~\ref{rem:anova}): the best non-decreasing predictor of $\1\{|2X-1|\ge t\}$ reproduces the increasing branch on $[\frac{1+t}2,1]$ exactly and replaces the decreasing branch on $[0,\frac{1+t}2]$ by the constant $\frac{1-t}{1+t}$. The decreasing branch contributes $\xid(C)$ in the same way, and the remaining share $\xi^\circ(C)=17-24\log2\approx0.364$ of Proposition~\ref{prop:three} requires a genuinely non-monotone predictor. For the transposed copula, $\xi(C^\top)=\frac14$ and $\xiu(C^\top)=\frac18$, so $\omega(C)^2=\frac58$ and $\omu(C)^2=\frac12(12\log2-8+\frac18)\approx0.221$; see Section~\ref{sec:sym}.
\end{example}

Table~\ref{tab:values} collects numerical values of $\xi$, $\xiu$ and $\xid$ for these examples.

\begin{table}[t]
\centering
\begin{tabular}{lccc}
\toprule
Copula & $\xi$ & $\xiu$ & $\xid$\\
\midrule
FGM, $\theta=0.8$ & $0.0427$ & $0.0427$ & $0$\\
Gaussian, $\rho=0.7$ & $0.3027$ & $0.3027$ & $0$\\
Gaussian, $\rho=-0.6$ & $0.2141$ & $0$ & $0.2141$\\
$\frac34M+\frac14W$ & $0.4375$ & $0.3949$ & $0.0044$\\
$\frac12M+\frac12W$ & $0.25$ & $0.0794$ & $0.0794$\\
Sawtooth, $k=8$ & $1$ & $0.125$ & $0$\\
$(U,|2U-1|)$ & $1$ & $0.3178$ & $0.3178$\\
\bottomrule
\end{tabular}
\caption{Values of $\xi$, $\xiu$ and $\xid$ for the examples of Section~\ref{sec:examples}.}
\label{tab:values}
\end{table}

\section{The symmetric version: mutual monotone dependence}\label{sec:sym}

By \eqref{eq:geom}, the full Sobolev geometry of \citet{SiburgStoimenov2008} is the orthogonal sum of the two directed geometries. This yields a symmetric counterpart of the decomposition in Theorem~\ref{thm:xiu}. Let $\CSI^\top=\{D^\top:D\in\CSI\}$ be the set of copulas for which $U$ is stochastically increasing in $V$; note that $d_2(C,\CSI^\top)=d_1(C^\top,\CSI)$.

\begin{definition}
For $C\in\C$ let $\omu(C)=\bigl(\tfrac12\,[\,\xiu(C)+\xiu(C^\top)\,]\bigr)^{1/2}$.
\end{definition}

\begin{theorem}\label{thm:sym}
For all $C,C_1,C_2,C_n\in\C$ the following hold.
\begin{enumerate}
\item[(a)] $\omega(C)^2=\omu(C)^2+3\bigl[d_1(C,\CSI)^2+d_2(C,\CSI^\top)^2\bigr]$; equivalently, in terms of the Sobolev norm,
\[
\|C-\Pi\|^2=\tfrac13\,\omu(C)^2+d_1(C,\CSI)^2+d_2(C,\CSI^\top)^2 .
\]
\item[(b)] $\omu(C)=\omega(C)$ if and only if $C\in\CSI\cap\CSI^\top$.
\item[(c)] $\omu(C)=0$ if and only if $C$ is NQD.
\item[(d)] $\omu(C)=1$ if and only if $C=M$.
\item[(e)] If $C_1\le C_2$, then $\omu(C_1)\le\omu(C_2)$.
\item[(f)] If $C_n\to C$ uniformly, then $\omu(C_n)\to\omu(C)$.
\end{enumerate}
\end{theorem}

\begin{proof}
Transposition preserves the concordance order, uniform convergence, NQD and $M$, and $\Proj(C^\top)$ is the $d_1$-projection of $C^\top$ onto $\CSI$. Hence all statements follow from Theorem~\ref{thm:xiu} applied to $C$ and $C^\top$, together with \eqref{eq:geom}. For (c), note that $C\le\Pi$ if and only if $C^\top\le\Pi$.
\end{proof}

\begin{remark}
In terms of the original problem of \citet{SiburgStoimenov2010}, $\omega$ measures how close $C$ is to an invertible copula, i.e., to a bijective functional relationship, while $\omu$ measures the part of this closeness that is explained by an increasing relationship in both directions. Theorem~\ref{thm:sym} combines two projections, $\Proj C$ and $(\Proj C^\top)^\top$, which in general are different copulas. The metric projection of $C$ onto $\CSI$, or onto $\CSI\cap\CSI^\top$, with respect to the \emph{full} Sobolev distance $d$ does not act section by section, because a change of $\partial_1C$ also changes $\partial_2C$. It leads to an obstacle-type variational problem that we do not pursue here.
\end{remark}

\section{Arbitrary distributions and estimation}\label{sec:general}

\subsection{Arbitrary distributions}

Let $(X,Y)$ be a random vector on $(\Omega,\mathcal F,\Prob)$, with $Y$ non-degenerate with distribution $\mu$. Put $f_t=\1\{Y\ge t\}$, $t\in\R$, and let $\K_X$ be the closure in $L^2(\Prob)$ of
\[
\bigl\{k(X):\ k\colon\R\to\R\ \text{non-decreasing and bounded}\bigr\}.
\]
This is a closed convex cone containing all constants, and $\K_X\subset L^2(\sigma(X))$. We define $\K_X$ as a closure because the set of all $k(X)\in L^2(\Prob)$ with $k\colon\R\to\R$ non-decreasing need not be closed: if $X$ is uniformly distributed on $(0,1)$, then $\log X$ is the $L^2$-limit of $\log(\max\{X,1/m\})$, but $\log X$ does not coincide almost surely with $k(X)$ for any non-decreasing $k\colon\R\to\R$, because such a $k$ would have to satisfy $k(0)\le\inf_{x\in(0,1)}\log x=-\infty$. For bounded elements, however, no such problem occurs.

\begin{lemma}\label{lem:KX}
Let $Z\in\K_X$ with $a\le Z\le b$ almost surely for some $a<b$. Then $Z=k(X)$ almost surely for a non-decreasing $k\colon\R\to[a,b]$.
\end{lemma}

\begin{proof}
Let $k_m(X)\to Z$ in $L^2(\Prob)$ with bounded non-decreasing $k_m$. The map $x\mapsto(x\vee a)\wedge b$ is $1$-Lipschitz and fixes $Z$, so we may assume that $k_m$ takes values in $[a,b]$. By a diagonal argument, a subsequence, still denoted by $(k_m)$, converges at every rational point. The function $\tilde k(x)=\sup_{q<x,\,q\in\mathbb Q}\lim_mk_m(q)$ is non-decreasing, and by monotonicity $k_m(x)\to\tilde k(x)$ at every continuity point $x$ of $\tilde k$. Since $\tilde k$ has at most countably many discontinuities, a further diagonal argument yields a subsequence that converges at every point of $\R$ (this is Helly's selection theorem in the form giving convergence everywhere). Its limit $k\colon\R\to[a,b]$ is non-decreasing. Along this subsequence, $k_m(X)\to k(X)$ pointwise and $k_m(X)\to Z$ in $L^2(\Prob)$, hence $Z=k(X)$ almost surely.
\end{proof}

Let $\Pi^\uparrow_X$ denote the metric projection onto $\K_X$. Since $\K_X$ is contained in the closed subspace $L^2(\sigma(X))$ and $\E[\,\cdot\mid X]$ is the orthogonal projection onto this subspace, we have $\Pi^\uparrow_Xf=\Pi^\uparrow_X\E[f\mid X]$. Since $\K_X$ is a closed convex cone containing $\pm1$, the characterisation \eqref{eq:moreau} and the arguments in the proof of Lemma~\ref{lem:iso}(ii) apply verbatim and give $\E\,\Pi^\uparrow_Xf=\E f$ and
\begin{equation}\label{eq:pyth-general}
\Var\bigl(\E[f\mid X]\bigr)=\Var\bigl(\Pi^\uparrow_Xf\bigr)+\bigl\|\E[f\mid X]-\Pi^\uparrow_Xf\bigr\|_{L^2}^2 .
\end{equation}
The map $t\mapsto f_t\in L^2$ is monotone and left-continuous, and $\Pi^\uparrow_X$ is continuous, so all integrands below are measurable.

\begin{definition}\label{def:general}
The \emph{monotone part} of Chatterjee's rank correlation is
\[
\xiu(Y|X)=\frac{\int\Var\bigl(\Pi^\uparrow_Xf_t\bigr)\,d\mu(t)}{\int\Var(f_t)\,d\mu(t)},
\]
and $\xid(Y|X)=\xiu(Y|-X)$. We say that $Y$ is SI in $X$ if for every $t\in\R$ there is a non-decreasing $k_t\colon\R\to[0,1]$ with $\Prob(Y\ge t\mid X)=k_t(X)$ almost surely.
\end{definition}

If $X$ takes only finitely many values, $\Pi^\uparrow_Xf_t$ is the weighted isotonic regression of $\Prob(Y\ge t\mid X=x)$ on the atoms of $X$, with weights $\Prob(X=x)$. The following two lemmas handle the passage from $\mu$-almost every $t$ to every $t$, and from conditional statements to functional ones.

\begin{lemma}\label{lem:allt}
Let $\Phi\colon L^2(\Prob)\to L^2(\Prob)$ be continuous and $S\subset L^2(\Prob)$ closed with $\Phi(c)\in S$ for all constants $c\in\{0,1\}$. If $\Phi(f_t)\in S$ for $\mu$-almost every $t$, then $\Phi(f_t)\in S$ for every $t\in\R$.
\end{lemma}

\begin{proof}
Let $T_0$ be the set of all $t$ with $\Phi(f_t)\in S$, so that $\mu(\R\setminus T_0)=0$, and fix $t\in\R$. If $\mu(\{t\})>0$, then $t\in T_0$. If $\mu((t-\varepsilon,t))>0$ for every $\varepsilon>0$, choose $t_n\in T_0$ with $t_n\uparrow t$; then $f_{t_n}\downarrow f_t$ pointwise, hence in $L^2$. If $\mu(\{t\})=0$ and $\mu((t,t+\varepsilon))>0$ for every $\varepsilon>0$, choose $t_n\in T_0$ with $t_n\downarrow t$; then $f_{t_n}\uparrow\1\{Y>t\}=f_t$ almost surely. Otherwise $t\notin\operatorname{supp}\mu$, and $f_t$ equals almost surely either a constant or $f_{t'}$ with $t'=\min(\operatorname{supp}\mu\cap[t,\infty))$, where $t'$ falls under one of the previous cases. In all cases $\Phi(f_t)\in S$ by continuity and closedness.
\end{proof}

\begin{lemma}\label{lem:monfun}
Let $Y=\varphi(X)$ almost surely for a measurable $\varphi$, and let $(X',Y')$, $(X'',Y'')$ be independent copies of $(X,Y)$. If $\Prob(X'<X'',\,Y'>Y'')=0$, then $Y=f(X)$ almost surely for a non-decreasing $f\colon\R\to[-\infty,\infty]$ that is finite $\Prob_X$-almost surely.
\end{lemma}

\begin{proof}
Let $f(x)=\inf\{c\in\R:\Prob(X\le x,\ \varphi(X)>c)=0\}$, the essential supremum of $\varphi$ with respect to $\Prob_X$ on $(-\infty,x]$. Clearly, $f$ is non-decreasing, and for every $x$ with $-\infty<f(x)<\infty$ the infimum is attained, since $\Prob(X\le x,\ \varphi(X)>c')\uparrow\Prob(X\le x,\ \varphi(X)>c)$ as $c'\downarrow c$. For rational $c$, the set $\{f\le c\}$ is an interval $J$ unbounded to the left, and $\Prob(X\in J,\ \varphi(X)>c)=0$ by approximating $J$ from inside. Hence $\Prob(\varphi(X)>c\ge f(X))=0$ for all rational $c$, i.e., $\varphi(X)\le f(X)$ almost surely. Suppose that $\Prob(f(X)>c>\varphi(X))>0$ for some rational $c$. For every $x$ with $f(x)>c$ we have $\Prob(X'\le x,\ \varphi(X')>c)>0$, hence $\Prob(X'\le X'',\ \varphi(X')>c>\varphi(X''))>0$. On $\{X'=X''\}$ the event $\{\varphi(X')>c>\varphi(X'')\}$ is empty, so $\Prob(X'<X'',\ Y'>Y'')>0$, a contradiction. Hence $\varphi(X)=f(X)$ almost surely.
\end{proof}

\begin{theorem}\label{thm:general}
Let $Y$ be non-degenerate.
\begin{enumerate}
\item[(a)] $\xi(Y|X)=\xiu(Y|X)+\frac{\int\|\E[f_t\mid X]-\Pi^\uparrow_Xf_t\|^2\,d\mu(t)}{\int\Var(f_t)\,d\mu(t)}$; in particular, $0\le\xiu(Y|X)\le\xi(Y|X)\le1$.
\item[(b)] $\xiu(Y|X)=\xi(Y|X)$ if and only if $Y$ is SI in $X$.
\item[(c)] $\xiu(Y|X)=0$ if and only if $(X,Y)$ is NQD, i.e., $\Prob(X\le x,Y\le y)\le\Prob(X\le x)\Prob(Y\le y)$ for all $x,y$.
\item[(d)] $\xiu(Y|X)=1$ if and only if $Y=f(X)$ almost surely for a non-decreasing $f\colon\R\to[-\infty,\infty]$ (which is then finite $\Prob_X$-almost surely).
\item[(e)] $\xiu(h(Y)|g(X))=\xiu(Y|X)$ for all strictly increasing $g,h\colon\R\to\R$.
\item[(f)] If $X$ and $Y$ have continuous distribution functions and copula $C$, then $\xiu(Y|X)=\xiu(C)$.
\end{enumerate}
\end{theorem}

\begin{proof}
(a) Integrate \eqref{eq:pyth-general} with $f=f_t$ with respect to $\mu$ and use \eqref{eq:xi-general}.

(b) By (a), equality holds if and only if $\E[f_t\mid X]\in\K_X$ for $\mu$-almost every $t$. Lemma~\ref{lem:allt} with $\Phi=\E[\,\cdot\mid X]$ and $S=\K_X$ extends this to every $t$. Since $\E[f_t\mid X]$ takes values in $[0,1]$, Lemma~\ref{lem:KX} shows that $\E[f_t\mid X]\in\K_X$ if and only if $\E[f_t\mid X]=k_t(X)$ almost surely for a non-decreasing $k_t\colon\R\to[0,1]$.

(c) $\xiu(Y|X)=0$ if and only if $\Pi^\uparrow_Xf_t$ is constant, and hence equal to $\E f_t$ by mean preservation, for $\mu$-almost every $t$, which by Lemma~\ref{lem:allt} (with $\Phi=\Pi^\uparrow_X$ and $S$ the set of constants) holds if and only if it is constant for every $t$. Since $\K_X$ is a closed convex cone containing the constants, the constant $\E f_t$ is the projection of $f_t$ if and only if $\ip{f_t-\E f_t}{k-\E f_t}=\Cov(f_t,k)\le0$ for all $k\in\K_X$. Since $k\mapsto\Cov(f_t,k)$ is continuous on $L^2(\Prob)$, it suffices to consider the generating elements of $\K_X$. Thus $\xiu(Y|X)=0$ if and only if $\Cov(k(X),\1\{Y\ge t\})\le0$ for all bounded non-decreasing $k$ and all $t$. By \citet{Lehmann1966}, NQD is equivalent to $\Cov(a(X),b(Y))\le0$ for all bounded non-decreasing $a,b$, which gives one direction. Conversely, taking $k=\1\{\cdot>x\}$ and approximating $\1\{Y>y\}$ by $\1\{Y\ge t_n\}$, $t_n\downarrow y$, gives $\Prob(X>x,Y>y)\le\Prob(X>x)\Prob(Y>y)$ for all $x,y$, which is equivalent to NQD.

(d) If $Y=f(X)$ with $f$ non-decreasing, then $f_t=\1\{f(X)\ge t\}$ is a bounded non-decreasing function of $X$, hence $f_t\in\K_X$, so $\Pi^\uparrow_Xf_t=f_t$ and $\xiu(Y|X)=1$. Conversely, let $\xiu(Y|X)=1$. By (a) and (b), $\xi(Y|X)=1$ and $Y$ is SI in $X$. By \citet{Chatterjee2021}, $Y=\varphi(X)$ almost surely for a measurable $\varphi$. Then $\1\{\varphi(X)\ge t\}=\Prob(Y\ge t\mid X)=k_t(X)$ almost surely, with $k_t$ non-decreasing, for every $t$ (by (b)). For independent copies, $\Prob(X'<X'',\ Y'\ge t>Y'')\le\Prob(X'<X'',\ k_t(X')=1,\ k_t(X'')=0)=0$ for every $t$. Since $\{Y'>Y''\}=\bigcup_{q\in\mathbb Q}\{Y'\ge q>Y''\}$, we obtain $\Prob(X'<X'',\,Y'>Y'')=0$, and Lemma~\ref{lem:monfun} applies. The extended range of $f$ cannot be avoided: for $X$ uniformly distributed on $(0,1)$ and $Y=\log X$, every non-decreasing $f$ with $Y=f(X)$ almost surely must satisfy $f=-\infty$ on $(-\infty,0]$.

(e) If $k$ is bounded and non-decreasing, then so is $k\circ g$. Conversely, with the non-decreasing generalised inverse $g^\leftarrow(y)=\sup\{x\in\R:g(x)\le y\}\in[-\infty,\infty]$ we have $g^\leftarrow(g(x))=x$, and $k(X)=\bar k(g^\leftarrow(g(X)))$, where $\bar k$ is the extension of $k$ to $[-\infty,\infty]$ by its limits. Hence the generating sets, and thus $\K_{g(X)}$ and $\K_X$, coincide. Moreover, $\{h(Y)\ge h(s)\}=\{Y\ge s\}$, and the image of $\mu$ under $h$ is the distribution of $h(Y)$, so numerator and denominator are unchanged.

(f) Let $U=F(X)$ and $V=G(Y)$, which are uniformly distributed. Since $U=F(X)$ and $X=F^{-1}(U)$ almost surely, where $F$ and the quantile function $F^{-1}$ are non-decreasing, the generating sets of $\K_X$ and $\K_U$ coincide up to null sets, as in (e); hence $\K_X=\K_U$. Since $G$ is continuous, $f_t=\1\{V\ge G(t)\}$ almost surely, and the substitution $v=G(t)$ shows that the denominator equals $\int_0^1v(1-v)\,dv=\frac16$. Moreover, $\Prob(V\ge v\mid U)=1-\partial_1C(U,v)$ almost surely, since $\Prob(V=v)=0$. The map $\phi\mapsto\phi(U)$ is an isometry from $L^2(0,1)$ onto $L^2(\sigma(U))$, and it maps the cone of non-decreasing functions in $L^2(0,1)$, which is the closure of the bounded ones, onto $\K_U$. Since this cone is $-\K$, we obtain $\Pi^\uparrow_U(1-\partial_1C(U,v))=1-\I(\partial_1C(\cdot,v))(U)$. As $\I$ preserves the mean, the numerator equals $\int_0^1\int_0^1(\I\partial_1C(\cdot,v)-v)^2\,du\,dv$, and Theorem~\ref{thm:proj}(a) gives $\xiu(Y|X)=6\,d_1(\Proj C,\Pi)^2=\xiu(C)$.
\end{proof}

\begin{remark}[Variance decomposition]\label{rem:anova}
The three nested sets $\K_X\subset L^2(\sigma(X))\subset L^2(\Prob)$ yield an orthogonal decomposition of the variation of the indicators $f_t$. For every $t$,
\[
f_t-\E f_t=\bigl(\Pi^\uparrow_Xf_t-\E f_t\bigr)+\bigl(\E[f_t\mid X]-\Pi^\uparrow_Xf_t\bigr)+\bigl(f_t-\E[f_t\mid X]\bigr),
\]
where the first two terms are orthogonal by the Moreau decomposition underlying \eqref{eq:pyth-general}, and the third term is orthogonal to $L^2(\sigma(X))$ and hence to both of them. Taking squared norms, integrating with respect to $\mu$ and dividing by $\int\Var(f_t)\,d\mu(t)$ gives
\[
1=\xiu(Y|X)+\bigl(\xi(Y|X)-\xiu(Y|X)\bigr)+\bigl(1-\xi(Y|X)\bigr),
\]
where $1-\xi(Y|X)=\int\E\Var(f_t\mid X)\,d\mu(t)\big/\!\int\Var(f_t)\,d\mu(t)$. Thus $\xi$ is the proportion of the variation of the conditional distribution functions that is explained by $X$, in the sense of a coefficient of determination for distributional regression, and $\xiu$ is the proportion explained by non-decreasing functions of $X$, i.e., by isotonic distributional regression. The difference $\xi-\xiu$ is the proportion explained by $X$ but not by non-decreasing functions of $X$, i.e., in a decreasing or genuinely non-monotone way. This decomposition holds for arbitrary distributions.
\end{remark}

\subsection{Estimation}\label{subsec:est}

Let $(X_1,Y_1),\dots,(X_n,Y_n)$ be independent copies of $(X,Y)$, where $X$ and $Y$ have continuous distribution functions $F$ and $G$ and copula $C$, and let $\Prob_n$ be the empirical distribution of the sample. Almost surely there are no ties, which we assume from now on. Since $\xiu(Y|X)$ is defined for every distribution, the plug-in estimator $\xiu(\Prob_n)$, i.e., $\xiu$ evaluated for a random vector with distribution $\Prob_n$, is available. For $\xi$ itself the plug-in principle fails: under $\Prob_n$, the conditional probabilities $\Prob_n(Y\ge t\mid X)$ are indicators, so that $\xi(\Prob_n)=1$ for every sample. Smoothing the empirical distribution into the empirical checkerboard copula $\widehat C_n$ introduced below does not help: a short calculation gives $\xi(\widehat C_n)=1-1/n$, although $\widehat C_n\to C$ uniformly, which reflects the lack of weak continuity of $\xi$ \citep{BucherDette2025}. Consistent estimators of $\xi$ instead rely on a nearest-neighbour construction \citep{Chatterjee2021,AzadkiaChatterjee2021} or on smoothing of copula derivatives \citep{DetteSiburgStoimenov2013}.

To describe $\xiu(\Prob_n)$, order the sample according to $X_{(1)}<\dots<X_{(n)}$, and let $S_i$ be the rank of the $Y$-value of the observation with $X$-value $X_{(i)}$ among $Y_1,\dots,Y_n$. For $k=0,\dots,n$ let $b_k\in\R^n$ be given by $(b_k)_i=\1\{S_i\le k\}$. We equip $\R^n$ with the norm $\|x\|_n^2=\frac1n\sum_{i=1}^nx_i^2$ and denote by $\I_n$ the metric projection onto the cone of non-increasing vectors, which is computed by the pool-adjacent-violators algorithm with equal weights; real numbers are identified with constant vectors. Finally, $\widehat C_n$ denotes the empirical checkerboard copula, which distributes the mass $1/n$ uniformly over each of the squares $[\frac{i-1}n,\frac in]\times[\frac{S_i-1}n,\frac{S_i}n]$, $i=1,\dots,n$ \citep{GenestNeslehovaRemillard2017}.

\begin{theorem}\label{thm:est}
Let $X$ and $Y$ have continuous distribution functions and copula $C$, and let $n\ge2$.
\begin{enumerate}
\item[(a)] For every sample without ties,
\[
\xiu(\Prob_n)=\frac{6n}{n^2-1}\sum_{k=1}^{n-1}\Bigl\|\I_nb_k-\frac kn\Bigr\|_n^2 .
\]
In particular, $\xiu(\Prob_n)$ depends only on the ranks of the sample, and it can be computed by $n-1$ runs of the pool-adjacent-violators algorithm, i.e., with $O(n^2)$ operations.
\item[(b)] For every sample without ties, $\bigl|\xiu(\Prob_n)-\xiu(\widehat C_n)\bigr|\le4/\sqrt n$.
\item[(c)] $\xiu(\Prob_n)\to\xiu(C)$ and $\xiu(\widehat C_n)\to\xiu(C)$ almost surely as $n\to\infty$.
\end{enumerate}
\end{theorem}

\begin{proof}
(a) Since $n\ge2$ and there are no ties, $Y$ is non-degenerate under $\Prob_n$, so $\xiu(\Prob_n)$ is defined. Under $\Prob_n$, the variable $X$ has the atoms $X_{(1)},\dots,X_{(n)}$, each with mass $1/n$. Hence $L^2(\Prob_n)$-functions of $X$ are identified with vectors in $(\R^n,\|\cdot\|_n)$, and $\K_X$ with the cone of non-decreasing vectors. Let $t$ be the $Y$-value of rank $k+1$, $k\in\{0,\dots,n-1\}$. Then $f_t=\1\{Y\ge t\}$ corresponds to the vector $1-b_k$, and $\Prob_n(Y\ge t)=1-k/n$. Since $x\mapsto1-x$ maps the non-increasing vectors onto the non-decreasing ones and is an affine isometry, $\Pi^\uparrow_Xf_t=1-\I_nb_k$. By the mean preservation established before \eqref{eq:pyth-general}, applied to $\Prob_n$, the vector $\Pi^\uparrow_Xf_t$ has mean $1-k/n$. Thus $\Var(\Pi^\uparrow_Xf_t)=\|\I_nb_k-k/n\|_n^2$ and $\Var(f_t)=\frac kn(1-\frac kn)$. Since the $Y$-marginal of $\Prob_n$ puts mass $1/n$ on each $Y$-value and the term with $k=0$ vanishes, Definition~\ref{def:general} gives
\[
\xiu(\Prob_n)=\frac{\frac1n\sum_{k=1}^{n-1}\|\I_nb_k-k/n\|_n^2}{\frac1n\sum_{k=0}^{n-1}\frac kn\bigl(1-\frac kn\bigr)}
=\frac{6n}{n^2-1}\sum_{k=1}^{n-1}\Bigl\|\I_nb_k-\frac kn\Bigr\|_n^2 ,
\]
because $\sum_{k=0}^{n-1}k(n-k)=n(n^2-1)/6$.

(b) For $u$ in the $i$-th strip $(\frac{i-1}n,\frac in)$ we have $\partial_1\widehat C_n(u,v)=h_i(v):=\min\{\max\{nv-S_i+1,0\},1\}$. The orthogonal projection $E$ of $L^2(0,1)$ onto the functions that are constant on each strip maps $\K$ into itself and satisfies $\|g-Ek\|=\|E(g-k)\|\le\|g-k\|$ whenever $g$ is constant on the strips. Hence, by uniqueness, the projection $\I g$ of such a $g$ is again constant on the strips, and it is given by $\I_n$ applied to the vector of strip values. Since $\|\cdot\|_{L^2(0,1)}$ reduces to $\|\cdot\|_n$ for such functions, Theorem~\ref{thm:proj}(a) gives
\[
\xiu(\widehat C_n)=6\int_0^1\bigl\|\I_nh(v)-v\bigr\|_n^2\,dv,\qquad h(v)=(h_1(v),\dots,h_n(v)).
\]
Let $v=(k+s)/n$ with $k\in\{0,\dots,n-1\}$ and $s\in[0,1]$, and let $m$ be the index with $S_m=k+1$. Then $h(v)=b_k+s\,e_m$, where $e_m$ is the $m$-th unit vector. Put $A=\I_n(b_k+s\,e_m)-v$ and $B=\I_nb_k-k/n$. Via this identification, $\I_n$ inherits the properties of $\I$: it commutes with the addition of constants, because $\K+c=\K$ for every constant $c$, and it is non-expansive by Lemma~\ref{lem:iso}(iv). Hence
\[
\|A-B\|_n\le\Bigl\|s\,e_m-\frac sn\Bigr\|_n=\frac{s\sqrt{n-1}}{n}\le\frac{s}{\sqrt n}.
\]
Moreover, the vectors $b_k+s\,e_m$ and $b_k$ have entries in $[0,1]$ and means $v$ and $k/n$, respectively; by Lemma~\ref{lem:iso}(ii), $\|A\|_n$ and $\|B\|_n$ are bounded by the corresponding standard deviations, hence by $\frac12$. Consequently $\bigl|\|A\|_n^2-\|B\|_n^2\bigr|\le(\|A\|_n+\|B\|_n)\|A-B\|_n\le s/\sqrt n$. Integrating over $v$, i.e., over $s$ with weight $1/n$, and summing over $k$, we obtain
\[
\Bigl|\xiu(\widehat C_n)-T_n\Bigr|\le\frac{3}{\sqrt n},\qquad T_n:=\frac6n\sum_{k=1}^{n-1}\Bigl\|\I_nb_k-\frac kn\Bigr\|_n^2 .
\]
By (a), $T_n=\frac{n^2-1}{n^2}\,\xiu(\Prob_n)$, so $|\xiu(\Prob_n)-T_n|\le n^{-2}\,\xiu(\Prob_n)\le n^{-2}$ by Theorem~\ref{thm:general}(a). Since $3/\sqrt n+n^{-2}\le4/\sqrt n$, (b) follows.

(c) Let $H_n$, $F_n$, $G_n$ be the empirical distribution functions of the sample and of its two coordinates, and $H$ the joint distribution function of $(X,Y)$. Let $Y_{(l)}$ denote the $l$-th smallest $Y$-value. For $k,l\in\{0,\dots,n\}$, $\widehat C_n(\frac kn,\frac ln)=\frac1n\#\{i\le k:S_i\le l\}=H_n(X_{(k)},Y_{(l)})$ (with the convention $X_{(0)}=Y_{(0)}=-\infty$), and $C(F(X_{(k)}),G(Y_{(l)}))=H(X_{(k)},Y_{(l)})$. Since $F_n(X_{(k)})=\frac kn$, $G_n(Y_{(l)})=\frac ln$, and $C$ is $1$-Lipschitz in each argument,
\[
\Bigl|\widehat C_n\Bigl(\frac kn,\frac ln\Bigr)-C\Bigl(\frac kn,\frac ln\Bigr)\Bigr|\le\|H_n-H\|_\infty+\|F_n-F\|_\infty+\|G_n-G\|_\infty .
\]
Both $\widehat C_n$ and $C$ are copulas, hence $1$-Lipschitz in each argument, so $\|\widehat C_n-C\|_\infty$ exceeds the maximum over the grid by at most $4/n$. By the Glivenko--Cantelli theorem, $\widehat C_n\to C$ uniformly almost surely. Theorem~\ref{thm:xiu}(f) yields $\xiu(\widehat C_n)\to\xiu(C)$ almost surely, and (b) yields the same for $\xiu(\Prob_n)$.
\end{proof}

\begin{remark}\label{rem:est}
Theorem~\ref{thm:est} holds for every copula, including singular ones such as the mixtures and shuffles of Section~\ref{sec:examples}. It requires neither smoothing nor a tuning parameter; the regularisation is provided entirely by the monotonicity constraint, which turns weak continuity (Theorem~\ref{thm:xiu}(f)) into consistency of the plug-in. More generally, $\xiu(\widetilde C_n)\to\xiu(C)$ almost surely for every copula estimator $\widetilde C_n$ that converges uniformly almost surely. The same applies to $\xid$, $\kappa$ and $\omu$. Note that $\xiu=0$ characterises negative quadrant dependence, not independence; a consistent estimator of a functional that vanishes exactly under independence is $\kappa(\Prob_n)=\max\{\xiu(\Prob_n),\xid(\Prob_n)\}$ (Corollary~\ref{cor:max}). The projection is also useful as a shape constraint: for every copula estimator $\widetilde C_n$, $\Proj\widetilde C_n$ is an SI copula dominating $\widetilde C_n$, and if $C\in\CSI$, then $\|\Proj\widetilde C_n-C\|_\infty=\|\Proj\widetilde C_n-\Proj C\|_\infty\le\|\widetilde C_n-C\|_\infty$ by Theorem~\ref{thm:proj}(e), so that consistency is preserved. The asymptotic distribution of $\xiu(\Prob_n)$ involves the directional differentiability of the least concave majorant operator \citep{BeareFang2017} and is beyond the scope of this paper; at the boundary $\xiu(C)=0$, where the isotonic regressions pool many values, a different rate than in the interior is to be expected, in line with the simulations below. Finally, since $b_k$ and $b_{k-1}$ differ in a single entry, an incremental implementation of the $n-1$ isotonic regressions may reduce the cost below $O(n^2)$; we do not pursue this here.
\end{remark}

\emph{Simulation study.} We illustrate Theorem~\ref{thm:est} for six copulas: the independence copula $\Pi$, the Gaussian copulas with $\rho=0.5$ and $\rho=-0.5$, the mixture $\frac12M+\frac12W$, the sawtooth shuffle with $k=8$, and the copula of $(U,|2U-1|)$. The last three are singular, and the Gaussian copula with $\rho=-0.5$ has $\xi=0.1447$ but $\xiu=0$. For each copula and each sample size $n\in\{100,500,2000,5000\}$ we drew $1000$ samples and computed $\xiu(\Prob_n)$ via Theorem~\ref{thm:est}(a). Table~\ref{tab:sim} reports the empirical means and standard deviations.

\begin{table}[t]
\centering
\small
\begin{tabular}{lccccc}
\toprule
Copula & $\xiu$ & $n=100$ & $n=500$ & $n=2000$ & $n=5000$\\
\midrule
$\Pi$ & $0.0000$ & $0.042$ $(0.023)$ & $0.012$ $(0.006)$ & $0.004$ $(0.002)$ & $0.002$ $(0.001)$\\
Gaussian, $\rho=0.5$ & $0.1447$ & $0.241$ $(0.050)$ & $0.180$ $(0.022)$ & $0.161$ $(0.011)$ & $0.153$ $(0.007)$\\
Gaussian, $\rho=-0.5$ & $0.0000$ & $0.002$ $(0.003)$ & $0.000$ $(0.000)$ & $0.000$ $(0.000)$ & $0.000$ $(0.000)$\\
$\frac12M+\frac12W$ & $0.0794$ & $0.130$ $(0.057)$ & $0.094$ $(0.022)$ & $0.084$ $(0.010)$ & $0.082$ $(0.006)$\\
Sawtooth, $k=8$ & $0.1250$ & $0.142$ $(0.039)$ & $0.129$ $(0.017)$ & $0.126$ $(0.008)$ & $0.126$ $(0.005)$\\
$(U,|2U-1|)$ & $0.3178$ & $0.324$ $(0.058)$ & $0.319$ $(0.027)$ & $0.319$ $(0.013)$ & $0.318$ $(0.009)$\\
\bottomrule
\end{tabular}
\caption{Mean and standard deviation (in parentheses) of $\xiu(\Prob_n)$ over $1000$ samples of size $n$; the second column gives the true value $\xiu(C)$.}
\label{tab:sim}
\end{table}

In all cases the estimates approach the true value as $n$ grows. Apart from negligible deviations, the bias is positive, which can be attributed to the fact that the isotonic regression partly fits the sampling noise; it is largest for the Gaussian copula with $\rho=0.5$ and small for the three singular copulas. For $\xiu(C)>0$, the standard deviations decrease approximately like $n^{-1/2}$. Under independence, mean and standard deviation decrease considerably faster, the mean approximately like $(\log n)/n$, and under the negatively dependent Gaussian copula the estimates are essentially zero from $n=500$ on, in line with Theorem~\ref{thm:general}(c). The singular copulas, for which estimators of $\xi$ based on copula derivatives would require smoothing, pose no difficulty.

\emph{Comparison with $\rho_S$ and $\xi$.} Table~\ref{tab:cmp} compares, for the same six copulas and $n=1000$, the empirical Spearman's rho $\hat\rho_S$, Chatterjee's estimator $\xi_n$ \citep{Chatterjee2021} and the plug-in estimators $\xiu(\Prob_n)$ and $\xid(\Prob_n)$, the latter computed from the sample $(-X_i,Y_i)$. The copulas form groups that one of the established measures cannot separate. The Gaussian copulas with $\rho=\pm0.5$ have the same value of $\xi$, while $\xiu$ and $\xid$ attribute the dependence entirely to one direction. The copulas $\Pi$, $\frac12M+\frac12W$ and that of $(U,|2U-1|)$ all have $\rho_S=0$, although the last two are clearly dependent and the last one is even functionally dependent. Conversely, the sawtooth shuffle and the copula of $(U,|2U-1|)$ both have $\xi=1$, but the sawtooth relationship has a small increasing and no decreasing part, whereas $(U,|2U-1|)$ has substantial increasing and decreasing parts. The estimators reflect these differences; the deviations from the true values are the finite-sample bias of $\xiu(\Prob_n)$ discussed above and, for the sawtooth shuffle, a bias of $\xi_n$ caused by the discontinuities of the relationship. The standard deviation of $\xiu(\Prob_n)$ is smaller than that of $\xi_n$ except for the two functional relationships, where $\xi_n$ is almost deterministic; the price of the monotonicity constraint is the positive finite-sample bias discussed above.

\begin{table}[t]
\centering
\small
\begin{tabular}{llcccc}
\toprule
Copula & & $\rho_S$ & $\xi$ & $\xiu$ & $\xid$\\
\midrule
$\Pi$ & true & $0$ & $0$ & $0$ & $0$\\
 & est. & $0.000$ $(0.032)$ & $-0.001$ $(0.020)$ & $0.007$ $(0.003)$ & $0.006$ $(0.003)$\\[2pt]
Gaussian, $\rho=0.5$ & true & $0.4826$ & $0.1447$ & $0.1447$ & $0$\\
 & est. & $0.483$ $(0.026)$ & $0.144$ $(0.023)$ & $0.169$ $(0.016)$ & $0.000$ $(0.000)$\\[2pt]
Gaussian, $\rho=-0.5$ & true & $-0.4826$ & $0.1447$ & $0$ & $0.1447$\\
 & est. & $-0.484$ $(0.025)$ & $0.144$ $(0.023)$ & $0.000$ $(0.000)$ & $0.169$ $(0.015)$\\[2pt]
$\frac12M+\frac12W$ & true & $0$ & $0.25$ & $0.0794$ & $0.0794$\\
 & est. & $-0.001$ $(0.042)$ & $0.249$ $(0.026)$ & $0.088$ $(0.015)$ & $0.088$ $(0.015)$\\[2pt]
Sawtooth, $k=8$ & true & $0.125$ & $1$ & $0.125$ & $0$\\
 & est. & $0.126$ $(0.032)$ & $0.956$ $(0.000)$ & $0.127$ $(0.011)$ & $0.002$ $(0.002)$\\[2pt]
$(U,|2U-1|)$ & true & $0$ & $1$ & $0.3178$ & $0.3178$\\
 & est. & $0.000$ $(0.041)$ & $0.994$ $(0.000)$ & $0.318$ $(0.018)$ & $0.318$ $(0.018)$\\
\bottomrule
\end{tabular}
\caption{True values and mean and standard deviation (in parentheses) of $\hat\rho_S$, $\xi_n$, $\xiu(\Prob_n)$ and $\xid(\Prob_n)$ over $1000$ samples of size $n=1000$.}
\label{tab:cmp}
\end{table}

\section*{Declaration of generative AI and AI-assisted technologies in the manuscript preparation process}

During the preparation of this work the author used Claude (Anthropic) to assist with literature search, \LaTeX\ and language editing of the text, checking calculations, and coding the simulation study. After using this tool, the author reviewed and edited the content as needed and takes full responsibility for the content of the published article.

\end{document}